\documentclass[smallextended]{svjour3}
\usepackage[utf8]{inputenc}
\usepackage{cite}
\usepackage{amsmath}
\usepackage{amsfonts}
\usepackage{amssymb}
\usepackage{subfigure}
\usepackage{enumerate}
\usepackage{hyperref}
\usepackage[colorinlistoftodos]{todonotes}
\usepackage{fullpage}

\newcommand {\R} {\mathbb R}
\newcommand {\E} {\mathbb E}
\newcommand {\B} {\mathbb B}
\newcommand {\Sp} {\mathbb S}
\newcommand{\al}{\alpha}
\newcommand{\be}{\beta}

\newcommand{\de}{\delta}
\newcommand{\deltabar}{{\overline{\delta}}}

\newcommand{\eps}{\varepsilon}
\newcommand{\bx}{\overline x}
\newcommand{\by}{\overline y}

\newcommand {\bd} {{\rm bd}\,}
\newcommand{\iv}{^{-1} }
\newcommand{\rg}{{\rm r}[A,B](\bx)}
\newcommand{\rgg}{{\rm r}'[A,B](\bx)}
\newcommand{\rga}{{\rm r}_a[A,B](\bx)}
\newcommand{\rgd}{{\rm r}_d[A,B](\bx)}
\newcommand{\rgdd}{{\rm r}_v[A,B](\bx)}
\newcommand{\sr}{{\rm sr}[A,B](\bx)}
\newcommand{\srr}{{\rm sr}'[A,B](\bx)}

\newcommand{\rgPhi}{{\rm r}[F](\bx,\by)}
\newcommand{\srgPhi}{{\rm sr}[F](\bx,\by)}

\newcommand {\Int} {{\rm int}\,}
\newcommand{\ang}[1]{\left\langle #1 \right\rangle}
\newcommand{\qdtx}[1]{\quad\mbox{#1}\quad}
\newcommand\xqed{%
  \leavevmode\unskip\penalty9999 \hbox{}\nobreak\hfill
  \quad\hbox{$\triangle$}}

\def\RHS{right-hand side}
\def\SVM{set-valued mapping}

\smartqed

\renewcommand{\equiv}{:=}
\newcommand{\Ball}{{\mathbb{B}}}

\newcommand{\Euclid}{\mathbb{E}}

\newcommand{\vbar}{{\overline{v}}}

\newcommand{\xbar}{{\overline{x}}}
\newcommand{\ybar}{{\overline{y}}}

\newcommand{\paren}[1]{\left(#1\right)}

\newcommand{\set}[2]{\left\{#1\,\left|\,#2\right.\right\}}
\newcommand{\ip}[2]{\left\langle #1,#2\right\rangle}
\newcommand{\norm}[1]{\left\|#1\right\|}

\newcommand{\mmap}[3]{#1:\,#2\rightrightarrows #3\,}

\newcommand{\TO}[1]{\stackrel{#1}{\to}}
\DeclareMathOperator{\reg}{reg}
\DeclareMathOperator{\subreg}{subreg}

\newcommand {\Limsup} {\mathop{{\rm Lim\,sup}\,}}

\DeclareMathOperator{\argmin}{argmin\,}

\DeclareMathOperator{\cone}{{cone}}

\DeclareMathOperator{\gph}{gph}

\DeclareMathOperator{\intr}{int\,}

\DeclareMathOperator{\Find}{\mathsf{Find}\,}

\newcommand{\nsub}[1]{N_{#1}}
\newcommand{\ncone}[1]{{N}_{#1}}
\newcommand{\fncone}[1]{N^F_{#1}}
\newcommand{\pncone}[1]{N^{\text{\rm prox}}_{#1}} 
\newcommand{\rpncone}[2]{N^{\text{\rm prox}}_{#1|#2}}
\newcommand{\rfncone}[2]{N^F_{#1|#2}}
\newcommand{\rncone}[2]{{N}_{#1|#2}} 

\newcommand{\cl}{\mbox{\rm cl}\,}

\newcommand{\sd}{\partial}

\title{Set Regularities and Feasibility
Problems}
\author{
Alexander Y. Kruger,
\thanks{Centre for Informatics and Applied Optimization, Federation University Australia, POB 663, Ballarat,
VIC 3350, Australia.  AYK was supported by Australian Research Council, project DP160100854.
E-mail:  \texttt{a.kruger@federation.edu.au} }
D. Russell Luke\thanks{Institut f\"ur Numerische und Angewandte Mathematik,
Universit\"at G\"ottingen,
37083 G\"ottingen, Germany. DRL was supported in part by German Israeli Foundation Grant G-1253-304.6 and 
Deutsche Forschungsgemeinschaft Research Training Grant 2088 TP-B5.
E-mail:  \texttt{r.luke@math.uni-goettingen.de}}
~and
Nguyen H. Thao\thanks{Institut f\"ur Numerische und Angewandte Mathematik,
Universit\"at G\"ottingen,
37083 G\"ottingen, Germany. NHT was supported by German Israeli Foundation Grant G-1253-304.6.
E-mail:  \texttt{h.nguyen@math.uni-goettingen.de} }
}
\begin{document}

\maketitle

\begin{abstract}
We synthesize and unify notions of regularity, both of individual sets and of collections of sets, as they appear
in the convergence theory of projection methods for consistent feasibility problems.  Several new
characterizations of regularities are presented which shed light on the relations between seemingly different
ideas and point to possible necessary conditions for local linear convergence of fundamental algorithms.\\
\end{abstract}

{\small \noindent {\bfseries 2010 Mathematics Subject
Classification:} {Primary 49J53, 65K10
Secondary 49K40, 49M05, 49M37, 65K05, 90C30.\\
}}


\section{Introduction}
In recent years there has been a tremendous interest in {\em first-order methods} for
solving variational problems.  As the name suggests, these methods only use information that,
in some way, encodes the gradient of a function to be minimized.  Often one has in mind
the following universal optimization problem for such methods
\begin{equation}\label{e:basic problem}
 \underset{x\in \Euclid}{\mbox{minimize}} \sum_{j=1}^m f_j(x)
\end{equation}
where $f_j$ are scalar extended-valued functions, not necessarily smooth or convex, on the
Euclidean space $\Euclid$.  This specializes to constrained optimization in the
case that one or more of the functions $f_j$ is an indicator function for a set.

All of the (nonconvex) convergence results concerning local linear convergence
that we have seen in the literature
involve implicitly or explicitly assumptions on the regularity of the functions $f_j$ and on
the relation of the functions to each other at critical points.  
Prominent
examples of this are the assumption that the functions satisfy the {\em Kurdyka-{\L}ojasiewicz}
property \cite{Loj63,Kur98,BolDanLew06,BolDanLeyMaz10} or that certain {\em constraint qualifications} are satisfied at
critical points, like interiority of
the constraints and invertability of the Jacobian of the objective in directions normal to the
constraints.

Our goal in this note is to identify common ideas and properties for a special case of
\eqref{e:basic problem} and to develop a general framework that both encapsulates
all of these ideas and is robust enough to be applied in other settings.
We focus our attention on the {\em feasibility} problem
\begin{equation}\label{e:feasibility}
 \Find \xbar\in \cap_{j=1}^m A_j,
\end{equation}
which is the specialization of \eqref{e:basic problem} to the case
\[
f_j(x)=\iota_{ A_j}(x)\equiv \begin{cases}0&\mbox{ if }x\in A_j\\ +\infty&\mbox{ else}\end{cases}\qquad (j=1,2,\dots,m).
\]

In the setting of the feasibility problem, regularity properties of individual sets (\emph{elemental regularity}
-- see Section~\ref{Reg_of_sets}) and of their intersections (\emph{transversality} and \emph{subtransversality} 
of collections of sets
-- see Section~\ref{Reg_of_Col}) come into play.
Both types of regularity have long history.
A typical classical elemental regularity assumption is the convexity of the sets, while the traditional assumption
that sets have a point in common on their interiors provides an example of transversal regularity in the convex setting.
Another classical example of the second type of regularity is the concept of transversality of smooth manifolds.

In the last decade there has been a great deal of interest in extending the classical notions of regularity to 
include nonconvex and nonsmooth sets, motivated to a large extent by nonsmooth and nonconvex 
optimization and attendant subdifferential and coderivative calculus, optimality and stationarity 
conditions and convergence analysis of algorithms.
Examples of modern elemental regularity concepts include \emph{Clarke regularity} \cite{RocWet98},
\emph{prox-regularity} (Poliquin, Rockafellar, and Thibault \cite{PolRocThi00}),
\emph{super-regularity} (Lewis, Luke, and Malick \cite{LewLukMal09}),
\emph{$(A',\varepsilon,\delta)$-re\-gularity} and \emph{$(\varepsilon,\delta)$-regularity}
(Bauschke, Luke, Phan, and Wang \cite{BauLukPhaWan13.2}),
\emph{relative ($A',\varepsilon,\delta$)-subregularity} (Hesse and Luke \cite{HesLuk13}),
\emph{relative $\sigma$-H\"older regularity} (Noll and Rondepierre \cite{NolRon15}).
Among the numerous transversal regularity (regularity of intersections) concepts we mention
\emph{Jameson properties (N) and (G)} \cite{Jam72}),
the \emph{conical hull intersection property (CHIP)} (Chui, Deutsch, and Ward \cite{ChuDeuWar90}),
(local, bounded) \emph{linear regularity} (Bauschke and Borwein \cite{BauBor93} and Zheng and Ng \cite{ZheNg08}),
the \emph{strong conical hull intersection property (strong CHIP)} (Deutsch, Li, and Ward \cite{DeuLiWar97}),
\emph{metric regularity} (Li \cite{Li97}),
\emph{metric inequality} (Ngai and Th\'era \cite{NgaThe01}),
the \emph{closed intersection property}, the \emph{normal property}, the \emph{weak}, \emph{uniform} and
\emph{dual normal properties}, the \emph{normal conical hull intersection property (normal CHIP)}
(Bakan, Deutsch, and Li \cite{BakDeuLi05}),
the {\em normal qualification condition} (Mordukhovich \cite{Mor06}),
\emph{regularity}, \emph{strong regularity}, or \emph{uniform regularity} (Kruger \cite{Kru05,Kru06,Kru09}),
the \emph{linearly regular intersection} (Lewis, Luke, and Malick \cite{LewLukMal09}),
\emph{linear coherence, alliedness} (Penot \cite{Pen13}),
the \emph{$(A,B)$-qualification condition} (Bauschke, Luke, Phan, and Wang \cite{BauLukPhaWan13.2}),
\emph{inherent and intrinsic transversality}
(Drusvyatskiy, Ioffe, and Lewis \cite{DruIofLew14, DruIofLew15}),
\emph{separable intersection} (Noll and Rondepierre \cite{NolRon15}),
\emph{transversality} and
\emph{subtransversality} (Ioffe \cite{Iof15}).
Some of the above concepts used under different names by different authors actually coincide.

A short survey of the recent developments in this area in the general nonsmooth and nonconvex
setting with the emphasis on convergence analysis is provided in Section~\ref{S2}.
The elemental and transversal regularity properties are further studied in Sections~\ref{Reg_of_sets}
and \ref{Reg_of_Col}, respectively.

Section~\ref{Reg_of_sets} introduces a general framework for elemental regularity of sets that
provides a common language for the many different definitions that have appeared to date.  This
new framework makes the cascade of implications between the different types of regularity
more transparent, namely that
convexity $\implies$  prox-regularity $\implies$ super-regularity $\implies$ Clarke regularity
$\implies$ $(\varepsilon,\delta)$-regularity $\implies$ $(\varepsilon,\delta)$-sub\-regularity
$\implies$ $\sigma$-H\"older regularity see Theorem \ref{t:set regularity equivalences}.
The last of these implications is new.

Section~\ref{Reg_of_Col} focuses on two local regularity properties of collections of 
sets which we call here \emph{subtransversality} and \emph{transversality}.
Both properties admit several representations and characterizations: metric, dual, angle, 
etc, and because of that or just historically are known under various names, e.g. (local) 
\emph{linear regularity}, \emph{metric regularity}, \emph{linear coherence}, and
\emph{metric inequality} for the first property, and \emph{strong regularity}, 
\emph{uniform regularity}, \emph{alliedness}, \emph{normal qualification condition} 
for the second one.  They correspond (are in a sense equivalent) to \emph{subregularity} 
and \emph{metric regularity} of \SVM s, respectively.
At the same time, the properties are related to certain `good' mutual arrangement of several 
objects (sets) in space, and after a discussion with experts in 
variational analysis\footnote{We are particulary indebted to Alex Ioffe for thoughtful and persuasive
discussions.}, 
we have decided to 
adopt the classical `transversality' terminology.

For ease of exposition, our discussion is limited to the case of just two closed subsets of Euclidean
space with nonempty intersection, though most of the regularity properties discussed are easily extended to
collections of more than two sets 
with nonempty intersection.
We compare the various representations of these properties and discuss also some recently introduced 
`restricted' regularity properties.
A number of characterizations of the properties are formulated.
Some characterizations are new, see Theorem~\ref{t:MAINi}\eqref{t:MAINid}, Theorem \ref{t:svm}\eqref{NewThe1i},
Theorem~\ref{t:MAINii}\eqref{t:MAINiic} and \eqref{t:MAINiig}, Theorem~\ref{t:tsr suff}\eqref{t:tsr suff iv} and
\eqref{t:separable and hoelder}.  We emphasize the important dual characterization of subtransversality in
Theorem~\ref{t:tsr suff}\eqref{t:tsr suff iv}
which expands and improves (in the setting adopted in the current article)
\cite[Theorem~4.1]{KruTha15}.
The proof of this assertion is going to appear in the forthcoming article \cite{KruLukTha}.
In contrast to dual characterizations for transversality which are necessary and sufficient, this note underscores
the fact that the known dual characterizations for subtransversality are {\em sufficient}.  This raises
the question whether {\em necessary} dual characterizations exist.

%
\subsection{Notation and basic definitions}
The {\em projector} onto the set $A$, $\mmap{P_{A}}{\Euclid}{A}$,
is central to algorithms for feasibility and is defined by
\[
P_{A}(x)\equiv \underset{a\in A}{\argmin}\norm{a-x}.
\]
A {\em projection} is a selection from the projector.
This exists for any closed set $ A$ in Euclidean space, as can be deduced by the
continuity and coercivity of the norm.  Note that the projector is not, in general, single-valued,
and indeed uniqueness of the projector defines a type of regularity of the set $ A$:  local
uniqueness characterizes {\em prox-regularity} \cite{PolRocThi00}  while in finite dimensional
settings global uniqueness characterizes convexity \cite{Bun34}.
The inverse of the projector is well defined:
\[
P_A^{-1}(a)\equiv\set{x\in\Euclid}{a\in P_A(x)}.
\]
Following \cite{BauLukPhaWan13.2}, we use this object to define the various normal cone mappings,
which in turn lead to  the
{\em subdifferential} of the indicator function $\iota_ A$.  This brings the theory presented here
to the edge of a much broader context of descent methods for solving
\eqref{e:basic problem}.  We will, however, focus exclusively on the feasibility problem for two sets.

Given a subset $A\subset\mathbb{E}$ and a point $\bx\in  A$, the
\emph{Fr\'echet, proximal and limiting normal cones} to $ A$ at $\bx$ are defined, respectively, as follows:
\begin{gather}\label{NC1}
\fncone{A}(\bx):= \left\{v\in\E\mid \limsup_{x\TO{ A}\bx,\,x\neq \bx} \frac {\langle v,x-\bx \rangle}{\|x-\bx\|} \leq 0 \right\},
\\\label{NC2}
\pncone{ A}(\bx):=\cone\left(P_{A}^{-1}(\bx)-\bx\right),
\\\label{NC3}
\ncone{A}(\xbar ):= \Limsup_{x\TO{A}\bx}\pncone{ A}(x):=\left\{v=
\lim_{k\to\infty}v_k\mid v_k\in \pncone{ A}(x_k),\; x_k\TO{ A} \bx \right\}.
\end{gather}
In the above and throughout this paper, $x\TO{ A} \bx$ means that $x\to \bx$ with $x\in  A$.

All these three sets are clearly cones.
Unlike the first two cones, the third one can be nonconvex.
It is easy to verify that
$\pncone{A}(\bx)\subseteq \fncone{A}(\bx)$.
Furthermore, if $A$ is closed, then
\begin{gather}\label{NC4}
\ncone{A}(\xbar )= \Limsup_{x\TO{A}\bx}\fncone{A}(x).
\end{gather}
The last formula can serve as a definition of the limiting normal cone in general normed linear spaces
when the original definition \eqref{NC3} in terms of proximal normals is not applicable.
If $\bx\in\bd A$, then $\ncone{A}(\xbar )\ne\{0\}$.
If $A$ is a convex set, then all three cones \eqref{NC1}--\eqref{NC3} coincide and reduce to the
normal cone in the sense of convex analysis:
\begin{gather}\label{CNC}
\ncone{A}(\bx):
= \left\{v\in\E\mid \langle v,x-\bx \rangle \leq 0 \qdtx{for all} x\in A\right\}.
\end{gather}

In differential geometry it is more common to work with  the tangent space, but for smooth manifolds the
normal cone \eqref{NC1} (the same as \eqref{NC3}) is a subspace and dual to the tangent space. Following Rockafellar and Wets \cite[Example~6.8]{RocWet98}, we say that
a subset $ A\subset\mathbb{E}$ is a $k$-di\-mensional ($0<k<n:=\dim\mathbb{E}$) smooth manifold around a point
$\xbar \in  A$ if there are a neighborhood $U$ of $\xbar $ in $\mathbb{E}$ and a smooth
(i.e., of $\mathcal{C}^1$ class) mapping $F:U\to \mathbb{R}^m$ ($m:=n-k$) with $\nabla F(\xbar )$ of
full rank $m$ such that
$ A\cap U = \{x\in U\mid F(x)=0\}$.
The \emph{tangent space} to $ A$ at $\xbar $ is a linear approximation of $ A$ near $\bx$ and is given by
\begin{align*}
T_{ A}(\xbar ) := \left\{x\in \mathbb{E}\mid \nabla F(\xbar )x=0\right\}.
\end{align*}
The \emph{normal space} to $A$ at $\xbar $ is defined as the orthogonal complement of
$T_{ A}(\xbar )$ and can be written as
\begin{align}\label{NS}
\nsub{A}(\xbar ) := \left\{\nabla F(\xbar )^{*}y\mid y\in \mathbb{R}^m\right\}.
\end{align}
It is in a sense a dual space object.
If $ A$ is a smooth manifold, then cones \eqref{NC1}, \eqref{NC3} and \eqref{CNC} 
reduce to the normal space \eqref{NS}.

Normal cones are central to characterizations both of the regularity of individual sets as well as of the regularity (transversality)
of collections of sets.  
For collections of sets, 
when dealing with projection algorithms,
it is important to account for the {\em relation} of the sets to
each other and so the classical definitions of the normal cones above are too blunt for a refined numerical
analysis.  A typical situation:
two nonempty sets $A$ and $B$ such that the affine span of $A\cup B$ is not equal to the whole space (e.g., two
distinct intersecting lines in $\R^3$).
One would expect all projections to lie in this affine span and the convergence to depend only on the
mutual arrangement of the sets within the span.
However, the normals (of any kind) to this affine span are also normals to the sets.
They make a nontrivial subspace and this causes problems for the regularity conditions on
collections of sets discussed below.
In the context of algorithms, the only regularity conditions that are relevant are those that
apply to the space where the iterates lie.  In the case of algorithms like alternating projections,
this is often an affine subspace of dimension smaller than the space in which the problem is formulated,
as the example above illustrates.

That normals to one of the sets should take into account the location of the other set was first
recognized by Bauschke, Luke, Phan and Wang
\cite{BauLukPhaWan13.2,BauLukPhaWan13.1} and has been used by Drusvyatskiy, Ioffe and Lewis
\cite{DruIofLew14,DruIofLew15}
\footnote{We refer on several occasions to the preprint \cite{DruIofLew14} because some definitions
and results present there and used in the current article are not included in the published version
\cite{DruIofLew15}.}, and Noll and Rondepierre \cite{NolRon15}, leading to weaker
``restricted'' regularity conditions.  The most straightforward idea is to consider only those normals
to each of the sets which are directed towards the other set.
Given two closed sets $A$ and $B$ and a point $\bx\in\bd A\cap\bd B$, one can define
(see \cite{BauLukPhaWan13.2,BauLukPhaWan13.1}) the following restricted analogues
of the cones \eqref{NC1}--\eqref{NC3}:
\begin{gather*}\notag
\rfncone{A}{B}(\bx):=\fncone{A}(\bx)\cap\cone\left(B-\bx\right),
\\\label{Bprox}
\rpncone{A}{B}(\bx):=\cone\left((P_A^{-1}(\bx)\cap B)-\bx\right),
\\\label{BlFr}
\rncone{A}{B}(\bx):= \Limsup_{x\TO{A}\bx}\rpncone{A}{B}(x),
\end{gather*}
which are called, respectively, the \emph{$B$-Fr\'echet normal cone}, \emph{$B$-proximal normal cone}
and \emph{$B$-limiting normal cone} to $A$ at $\bx$.
When $B=\E$, they coincide with the cones \eqref{NC1}--\eqref{NC3}.
Note that in general $\rpncone{A}{B}(\bx)$ and $\rncone{A}{B}(\bx)$ can be empty.
The essence of what we call ``dual regularity conditions'' consists in
computing appropriate normal cones (limiting, Fr\'echet, or proximal) to each of the sets at the
reference point (or nearby) and ensuring that the cones do not contain oppositely directed nonzero vectors.
Such conditions are important for many applications including convergence analysis of projection algorithms.

The regularity/transversality properties of the collection $\{A,B\}$ in $\Euclid$
can be understood in terms of the corresponding properties of the set-valued mapping
$F:\Euclid \rightrightarrows {\Euclid}^2$
defined by (cf. \cite{Iof00,Iof15,DruIofLew15,LewLukMal09})
\begin{equation}\label{FAB}
F(x):= (A-x)\times (B-x).
\end{equation}
For $x \in \Euclid$ and $u=(u_1,u_2) \in {\Euclid}^2$, we have
$$
x \in A\cap B \iff (0,0) \in F(x),\quad F^{-1}(u) = (A-u_1)\cap (B-u_2),\qdtx{and} F^{-1}(0) =A\cap B.
$$
The mapping \eqref{FAB} is not the only possibility.
Another useful construction is given by the \SVM\ $G:\Euclid^2 \rightrightarrows \Euclid$
(cf. \cite[Page~226]{LewMal08}, \cite[Page~1638]{DruIofLew15}, \cite[Corollary~7.13]{Iof15}) defined as follows:
\begin{equation}\label{F3}
G(x_1,x_2):=
\begin{cases}
\{x_1-x_2\}&\mbox{if } x_1\in A\mbox{ and }x_2\in B,
\\
\emptyset&\mbox{otherwise}.
\end{cases}
\end{equation}
Obviously,
\begin{equation*}
0 \in G(x_1,x_2) \quad\Longleftrightarrow\quad x_1=x_2\in A\cap B.
\end{equation*}

Conversely, the regularity of certain set-valued mappings can be understood
in terms of properties of the corresponding properties of collections of sets.  Indeed,
 given a set-valued mapping $F:\E_1\rightrightarrows\E_2$,
its regularity properties a point $(\bx,\by)$ in its
graph $\gph F:=\{(x,y)\in\E_1\times\E_2\mid y\in F(x)\}$ are connected to those of
the collection of sets defined by (cf. \cite[Corollary~2.1]{Kru05})
\begin{equation*}
A:=\gph F\qdtx{and}B:=\E_1\times \{\bx_2\}
\end{equation*}
in $\E_1\times\E_2$ .
One can check that $(\bx_1,\bx_2)\in A\cap B=F\iv(\bx_2)\times \{\bx_2\}$.

In the next definition, we recall two common regularity properties for a set-valued mapping
$F:\E_1\rightrightarrows\E_2$ between two Euclidean spaces at a point $(\bx,\by)\in\gph F$.
\begin{definition}\label{MR}
\begin{enumerate}
\item\label{MRi}
$F$ is \emph{metrically subregular} at $\xbar $ for $\ybar$ if there exist $\delta>0$ and
$\al>0$
such that
\begin{equation}\label{msr}
\al d\left(x,F^{-1}(\ybar )\right) \le d(\ybar ,F(x))\;\;\mbox{for all}\;\; x \in \Ball_{\delta}(\xbar ).
\end{equation}
If, additionally, $\xbar $ is an isolated point of $F^{-1}(\ybar )$, then $F$ is called {\em strongly
metrically subregular at $\xbar $ for $\ybar $}.
\item\label{MRii}
$F$ is \emph{metrically regular} at $\xbar $ for $\ybar $ if there exist $\delta>0$ and
$\al>0$
such that
\begin{equation}\label{mr}
\al d\left(x,F^{-1}(y)\right) \le d\left(y,F(x)\right)\;\;\mbox{for all}\;\; x \in \Ball_{\delta}(\xbar ),\; y \in \Ball_{\delta}(\ybar ).
\end{equation}
\end{enumerate}
\end{definition}

We use the notation $\srgPhi$ and $\rgPhi$ to denote the supremum of all $\al$ such that conditions
\eqref{msr} and \eqref{mr}, respectively, hold for some $\de>0$.
Properties \eqref{MRi} and  \eqref{MRii} in Definition~\ref{MR} are equivalent to conditions
$\srgPhi>0$ and $\rgPhi>0$, respectively, and the values $\srgPhi$ and $\rgPhi$ characterize the
corresponding properties quantitatively.
Some authors use for that purpose the reciprocals of the above constants (cf. \cite{DonRoc09}):
$$\subreg(F;\xbar |\ybar ):=\frac{1}{\srgPhi},\quad
\reg(F;\xbar |\ybar ):=\frac{1}{\rgPhi},$$
which are referred to as \emph{subregularity modulus} and \emph{regularity modulus}, respectively.
One obviously has $0\le\rgPhi\le\srgPhi$, which means that metric regularity is in general a
stronger property than metric subregularity.

Both regularity properties in Definition~\ref{MR} are fundamental for variational analysis
(especially the second one) and have found numerous applications in optimization and other fields.
Several useful characterizations of these properties and the fundamental equivalences
\begin{gather*}
\mbox{Metric regularity}
\;\;\Leftrightarrow\;\;
\mbox{Linear openness}
\;\;\Leftrightarrow\;\;
\mbox{Aubin property of the inverse},
\\
\mbox{Metric subregularity}
\;\;\Leftrightarrow\;\;
\mbox{Linear subopenness}
\;\;\Leftrightarrow\;\;
\mbox{Calmness of the inverse}
\end{gather*}
have been established.
We refer the readers to the monographs \cite{Mor06,DonRoc09} and surveys
\cite{Aze06,Iof15,Iof00,ApeDurStr13} for a comprehensive exposition of the properties.
Note that metric subregularity of $F$ at $\bx$ for $\by\in F(\bx)$ is equivalent to the local error bound property of the
real-valued function $x\mapsto d(\ybar ,F(x))$ at $\bx$, while metric regularity of $F$ means that the mentioned
error bound property holds uniformly with respect to $y$ in a neighborhood of $\by$.

Our other basic notation is standard; cf. \cite{Mor06,RocWet98, DonRoc09}.
The open unit ball and the unit sphere in a Euclidean space are denoted $\B$ and $\mathbb{S}$, respectively.
$\B_\delta(x)$ stands for the open ball with radius $\delta>0$ and center $x$.
For a subspace $V$ of a Euclidean space $\E$, $V^{\perp}:=\set{u\in \E}{\ip{u}{v}=0 \mbox{ for all } v\in V}$ is the
orthogonal complement subspace of $V$.
For a real number $\alpha$, $[\alpha]_+$ denotes $\max\{\alpha,0\}$.

\section{Regularity notions and convergence results}\label{S2}
In the definitions below, we keep the terminology coming from the original publications although the
use of words obviously lacks consistency.

\subsection{Elemental set regularities}
The following definition is a compilation of some of the main definitions of regularities of sets.
\begin{definition}[set regularities]\label{d:set regularities}
 Let $ A\subset \Euclid$ be closed
and let the projector be with respect to the Euclidean norm.
\begin{enumerate}
  \item\label{d:Hoelder}\cite[Definition 2]{NolRon15} $ A$ is {\em $\sigma$-H\"older regular relative to
  $B\subset \E$ at $\xbar\in A\cap B$ with constant $c>0$}
  if there exists a neighborhood $W$ of $\xbar$ such that for each $b\in B\cap W$ and $b_A\in P_A(b)\cap W$,
  \begin{equation*}
   \Ball_{(1+c)\norm{b-b_A}}(b)\cap A\cap V =\emptyset,
  \end{equation*}
where
\[
V\equiv \set{x\in P^{-1}_B(b)}{\ip{b-b_A}{x-b_A}>\sqrt{c}\norm{b-b_A}^{\sigma+1}\norm{x-b_A}}.
\]

\item\label{d:L-eps-del-subregular}\cite[Definition 2.9]{HesLuk13}
 $A$ is \emph{($ A',\varepsilon,\delta$)-subregular relative to $B\subset\Euclid$ at $\bx$}
if
\begin{align}\label{eq:epsilondeltasubregularity}
\ip{v}{b-a} \leq\varepsilon\norm{v}\norm{b-a}
\end{align}
holds for all $b\in B\cap \Ball_\delta(\xbar)$,
$a\in A\cap \Ball_\delta(\xbar)$ and $v\in \rpncone{ A}{ A'}(a)$.
If $  B=\{\xbar\}$, then the respective qualifier, ``relative to'' is dropped.
If $ A'=\Euclid$, then $ A$ is said to be
\emph{$(\varepsilon,\delta)$-subregular} (relative to $  B$) at $\xbar$.

\item\label{d:L-eps-del-regular}\cite[Definition 8.1]{BauLukPhaWan13.2}
If $  B= A$ in \eqref{d:L-eps-del-subregular} above,
then the set $ A$ is said to be \emph{$( A', \varepsilon,\delta)$-regular} at
$\xbar$.
If also $ A'=\Euclid$, then $ A$ is said to be
\emph{$(\varepsilon,\delta)$-regular} at
$\xbar$.

\item\label{d:Clarke}\cite[Definition 6.4]{RocWet98} $ A$ is \emph{Clarke regular} at $\xbar\in A$
if $\ncone{A}(\bx)=\fncone{A}(\bx)$.

\item\label{d:super-reg}\cite[Definition 4.3]{LewLukMal09}
$ A$ is \emph{super-regular} at $\xbar\in A$ if for any $\varepsilon>0$, there is a $\delta>0$ such that
  \begin{equation*}
    \ip{z-z_A}{x-z_A}\leq \varepsilon\norm{z-z_A}\|x-z_A\|
\;\; \mbox{for all}\;\;
x\in  A\cap\Ball_\delta(\xbar),
    z\in \Ball_\delta(\xbar), z_A\in P_A(z).
  \end{equation*}

\item\label{d:prox-reg}\cite[Definition 1.1]{PolRocThi00}
$ A$ is \emph{prox-regular}
at $\xbar\in A$ for $\vbar\in \ncone{A}(\xbar)$ if there exist $\varepsilon,\delta>0$ such that
\begin{equation*}
\ip{v}{x-a}\leq \frac{\varepsilon}{2}\norm{x-a}^2
\;\; \mbox{for all}\;\;
x,a \in  A\cap \Ball_\delta(\xbar),
v\in \Ball_{\delta}(\vbar) \cap \ncone{A}(a).
\end{equation*}
If $ A$ is prox-regular
at $\xbar\in A$ for
all $v\in \ncone{A}(\xbar)$,
then $ A$ is said to be prox-regular at $\xbar$.
\end{enumerate}
\end{definition}
Convexity, of course, implies all properties in Definition \ref{d:set regularities} globally.

\subsection{Regularity of collections of sets}
The origins of the concept of regular arrangement of sets in space can be traced back to that of \emph{transversality} in
differential geometry which deals of course with smooth manifolds (see, for instance, \cite{GuiPol74,Hir76}).

Given two smooth manifolds $A,B\subset\mathbb{E}$ around a point $\xbar \in A\cap B$,
an important question is whether their intersection $A\cap B$ is also a smooth manifold around $\xbar $.
For that to be true, the manifolds must intersect in a certain regular way.
The typical sufficient regularity assumption is known as \emph{transversality}:
the collection $\{A,B\}$ of smooth manifolds is
{\em transversal} at $\xbar\in A\cap B $ if
\begin{equation}\label{Tran-1}
T_A(\xbar ) + T_B(\xbar ) = \Euclid.
\end{equation}
This notion has been used in \cite{LewMal08,GuiPol74,Iof15}.
Under this assumption,
$A\cap B$ is a smooth manifold around $\xbar $ and the
following equalities hold (cf. \cite{LewMal08,GuiPol74,Iof15}):
\begin{align}\label{BasQua-}
T_{A\cap B}(\xbar ) &= T_A(\xbar ) \cap T_B(\xbar ),\\
\nsub{A}(\xbar ) &\cap \nsub{B}(\xbar )= \{0\}. \label{BasQua}
\end{align}
Equality~\eqref{BasQua-} is only a necessary condition and is in general weaker than condition~\eqref{BasQua}.
When $A$ and $B$ are convex sets, it is known as
the \emph{conical hull intersection property (CHIP)} \cite{ChuDeuWar90} (cf. \cite[Definition~5.1]{BakDeuLi05}).

\begin{example}\label{E1}
Let $\mathbb{E}=\R^2$, $A=B=\R\times\{0\}$, $\bx=(0,0)$.
Then $A\cap B=A=B$, $T_{A\cap B}(\xbar ) = T_A(\xbar ) =T_B(\xbar )$, and consequently, equality~\eqref{BasQua-} holds true.
At the same time, $T_A(\xbar )+T_B(\xbar )=\R\times\{0\}$ and $\nsub{A}(\xbar )=\nsub{B}(\xbar )=\{0\}\times\R$.
Thus condition \eqref{BasQua} is not satisfied.
\xqed\end{example}

The transversality property of a collection of two smooth manifolds can be characterized quantitatively in
terms of the angle between their tangent (or normal) spaces.
One can use for that purpose the \emph{Friedrichs} angle (cf. Deutsch \cite{Deu01}).
\begin{definition}[Friedrichs angle]
Given two nonempty subspaces $V_1$ and $V_2$, the Friedrichs angle is a number between 0 and $\pi/2$
whose cosine is given by
\begin{align}\label{Fried}
c(V_1,V_2) &:= \max\left\{\ip{v_1}{v_2}\mid v_1\in V_1\cap (V_1\cap V_2)^{\perp}\cap\B,\;
v_2\in V_2\cap (V_1\cap V_2)^{\perp}\cap\B\right\}.
\end{align}
\end{definition}

The following properties provide some insight into the geometry of the intersection
(cf. \cite[Fact~7.10]{BauLukPhaWan13.2}, \cite[Lemmas 3.2 and 3.3]{LewMal08}):
\begin{gather}\label{Fri1}
c(V_1,V_2)=c(V_1^{\perp},V_2^{\perp})<1,\\
\label{Fri2}
1-c(V_1,V_2)=\min\left\{ d^2(v,V_1)+ d^2(v,V_2)\mid \norm{v}=1\right\}
\qdtx{if}V_1\cap V_2=\{0\}.
\end{gather}

The angle between two smooth manifolds $A$ and $B$ around a point $\xbar \in A\cap B$ is
defined in \cite[Definition~3.1]{LewMal08} as the Friedrichs angle between the two tangent
subspaces $T_A(\xbar )$ and $T_B(\xbar )$, or equivalently, in view of \eqref{Fri1}, the
Friedrichs angle between the two normal subspaces $\nsub{A}(\xbar )$ and $\nsub{B}(\xbar )$:
\begin{equation}\label{e:angle}
c(A,B,\xbar ) := c(T_A(\xbar ),T_B(\xbar ))= c(\nsub{A}(\xbar),\nsub{B}(\xbar )).
\end{equation}
Observe that, with $V_1=\nsub{A}(\bx)$ and $V_2=\nsub{B}(\bx)$, condition $V_1\cap V_2=\{0\}$ in
\eqref{Fri2} is equivalent to \eqref{BasQua}, and thus, to the transversality of $\{A,B\}$ at $\xbar $.

Expanding our scope to
the general case of two closed sets $A,B\subset\mathbb{E}$ having a common point $\xbar \in A\cap B$
the intuition behind transversality plays an important role, but it is clear that
a richer vocabulary is needed to describe the many ways in which sets
can intersect `transversally'.
\begin{definition}[regularities of collections of sets]\label{D4}
Suppose $A$ and $B$ are closed sets of $\E$ and $\bx\in\bd A\cap\bd B$.
\begin{enumerate}
\item\label{D4:lin-reg}  \cite[Page 99]{ClaLedSteWol98}  The {\em transversality condition}
holds at the intersection of two sets $A$ and $B$  if
\begin{equation}\label{e:TransCond}
\ncone{A}(\xbar ) \cap (-\ncone{B}(\xbar )) = \{0\}
\end{equation}
at $\xbar\in A\cap B$.
\item\label{D4:local lin reg} \cite[Page~62]{ZheNg08} The collection of sets $\{A,B\}$ is {\em locally linearly regular} at
$\bx\in A\cap B$ if there exist numbers $\delta>0$ and $\alpha>0$
such that
\[
\alpha d\left(x,A\cap B\right)\le \max\left\{d(x,A),d(x,B)\right\}
\;\;\mbox{for all}\;\; x\in \B_{\delta}(\xbar ).
\]

\item\label{D4:CQ} \cite[Definition~6.6]{BauLukPhaWan13.2}
The \emph{$(A,B)$-qualification condition}  holds at $\bx$
if one of the following equivalent formulations holds:
\begin{enumerate}
\item
there exists a number $\alpha<1$ such that
$-\ang{v_1,v_2}<\al$
for all $v_1\in\rncone{A}{B}(\bx)\cap\Sp$ and $v_2\in\rncone{B}{A}(\bx)\cap\Sp$; or
\item
there exist numbers $\alpha<1$ and $\delta>0$ such that
$-\ang{v_1,v_2}<\al$
for all $a\in A\cap \B_{\de}(\bx)$, $b\in B\cap \B_{\de}(\bx)$, $v_1\in \rpncone{A}{B}(a)\cap\Sp$, and
$v_2\in \rpncone{B}{A}(b)\cap\Sp$.
\end{enumerate}

\item\label{D4:DIL} \cite[Definition~4.4]{DruIofLew14}
$A$ and $B$ are \emph{inherently transversal}  at
$\bx$ if  there exist numbers $\alpha<1$ and $\delta>0$ such that
$$
\ang{a_1-b_2,a_2-b_1}<\al\|a_1-b_2\|\,\|a_2-b_1\|
$$
for all $a_1\in(A\setminus B)\cap \B_{\delta}(\bx)$, $b_1\in(B\setminus A)\cap \B_{\delta}(\bx)$, $b_2\in P_B(a_1)$,
and $a_2\in P_A(b_1)$.

\item\label{D4:NolRon} \cite[Definition~1]{NolRon15}
\emph{$B$ intersects $A$ separably} at $\bx$ if
there exist numbers $\alpha<1$ and $\delta>0$ such that
$$
\ang{a_1-b,a_2-b}<\al\|a_1-b\|\,\|a_2-b\|
$$
for all $a_1\in(A\setminus B)\cap \B_{\delta}(\bx)$, $b\in(P_B(a_1)\setminus A)\cap \B_{\delta}(\bx)$,
and $a_2\in P_A(b)\cap \B_{\delta}(\bx)$.

If $A$ also intersects $B$ separably at $\bx$, then $\{A,B\}$ is said to intersect separately at $\bx$.

\item\label{D4:DIL2} \cite[Definition~3.1]{DruIofLew15}
$A$ and $B$ are \emph{intrinsically transversal} at $\bx$ if one of the following equivalent conditions holds:
\begin{enumerate}
\item
there exist numbers $\alpha>0$ and $\delta>0$ such that
\begin{equation*}
\max\biggl\{d\left(\frac{b-a}{\|a-b\|}, \ncone{A}(a)\right),
d\left(\frac{a-b}{\|a-b\|}, \ncone{B}(b)\right)\biggr\}>\al
\end{equation*}
for all $a\in(A\setminus B)\cap \B_{\delta}(\bx)$ and $b\in(B\setminus A)\cap \B_{\delta}(\bx)$; or
\item
there exist numbers $\alpha>0$ and $\delta>0$ such that
\begin{equation*}
\max\biggl\{d\left(\frac{b-a}{\|a-b\|}, \pncone{A}(a)\right),
d\left(\frac{a-b}{\|a-b\|}, \pncone{B}(b)\right)\biggr\}>\al
\end{equation*}
for all $a\in(A\setminus B)\cap \B_{\delta}(\bx)$ and $b\in(B\setminus A)\cap \B_{\delta}(\bx)$.
\end{enumerate}
\end{enumerate}
\end{definition}

Using the Euclidean space geometry, each of the properties in Definition~\ref{D4} can be reformulated equivalently in several
different ways;
see some reformulations in \cite{KruTha16} including the angle
characterization of intrinsic transversality in \cite[Proposition~19]{KruTha16}.
Analytically, this means quantifying each of these properties.
As observed in \cite[Theorem 18]{LewMal08}, for two smooth manifolds the equalities \eqref{Tran-1} and
\eqref{BasQua} are actually equivalent.
Properties \eqref{D4:lin-reg} and
\eqref{D4:local lin reg} are shown in Theorems \ref{t:MAINi} and \ref{t:MAINii} below to be equivalent
to what we call in this article (Definition \ref{D2}) {\em transversality}
and {\em subtransversality}, respectively.

A more general H\"older-type setting of property \eqref{D4:NolRon} with exponent $\omega\in [0,2)$ is considered in \cite{NolRon15}.
Definition \ref{D4}\eqref{D4:NolRon} corresponds to the `linear' case $\omega=0$.  Note also that this
notion is not symmetric: $B$ may intersect $A$ separably, but $A$ need not intersect $B$ separably.

\begin{proposition}[relations between regularities of collections]\label{t:collection relations a}
The known relationships between the properties in Definition~\ref{D4} are as follows:
\begin{enumerate}[\rm (a)]
\item\label{t:collection relations ia}
\eqref{D4:CQ} $\Rightarrow$ \eqref{D4:DIL} $\Rightarrow$ \eqref{D4:NolRon}.
\item\label{t:collection relations iia}
\eqref{D4:DIL2} $\Rightarrow$ \eqref{D4:NolRon}.
\item\label{t:collection relations iva}
Property \eqref{D4:DIL2} is in general independent of each of the properties \eqref{D4:CQ} and \eqref{D4:DIL}.
\item\label{t:collection relations va}
When both sets are super-regular (Definition \ref{d:set regularities}\eqref{d:super-reg})
at the reference point, \eqref{D4:DIL} $\Rightarrow$ \eqref{D4:DIL2}.
\end{enumerate}
\end{proposition}
\begin{proof}
Implications {\em \eqref{t:collection relations ia}}
and {\em \eqref{t:collection relations iia}} follow from the description of the properties in
Definition \ref{D4}.
The observation  {\em \eqref{t:collection relations iva}} was demonstrated in \cite[Examples 23 and 24]{KruTha16}.
Implication {\em \eqref{t:collection relations va}} was shown in \cite[Proposition~4.5]{DruIofLew14}.
\qed
\end{proof}


\subsection{Convergence results}
We catalog next some of the main nonconvex convergence results making use of certain combinations of the
above regularities.
\begin{proposition}[linear convergence of alternating projections]\label{t:AP}
For $A$ and $B$ closed with nonempty intersection and $\bx\in A\cap B$, the alternating projections
algorithm converges locally linearly if one of the following collection of conditions holds.
 \begin{enumerate}[\rm (i)]
  \item\label{t:LM} \cite[Theorem 4.3]{LewMal08} $A$ and $B$ are smooth manifolds around $\xbar$ and
  $\{A,B\}$ is transversal at $\xbar$.
  \item\label{t:LLM} \cite[Theorem 5.16]{LewLukMal09} $A$ is super-regular at $\xbar$ and
 the transversality condition \eqref{e:TransCond}
holds at $\xbar$.
  \item\label{t:BLPW} \cite[Theorem 3.17]{BauLukPhaWan13.2} $A$ is $(B,\varepsilon,\delta)$-regular at
  $\xbar$ and the $(A,B)$-qualification condition holds at $\xbar$.
  \item\label{t:LH} \cite[Theorem 3.11]{HesLuk13} $A$ and $B$ are $(\varepsilon,\delta)$-subregular
  relative to $A\cap B$ at $\xbar$ and $\{A,B\}$
is locally linearly regular
at $\xbar$.
  \item\label{t:DIL} \cite[Theorem 2.3]{DruIofLew14} $\{A,B\}$ is intrinsically transversal
at $\xbar$.
  \item\label{t:NollRon} \cite[Theorem 2]{NolRon15}
$A$ is $0$-H\"older regular relative to $B$ at $\xbar$ and $\{A,B\}$ intersects separably at $\xbar$.
 \end{enumerate}
\end{proposition}

\begin{proposition}[linear convergence of the Douglas-Rachford algorithm]\label{t:DR}
 For $A$ and $B$ closed with nonempty intersection and $\bx\in A\cap B$, the Douglas-Rachford algorithm
 converges locally linearly if one of the following collection of conditions holds.
 \begin{enumerate}[\rm (i)]
  \item\label{t:LH_DR} \cite[Theorem 3.18]{HesLuk13}
$\{A,B\}$ satisfies the transversality condition \eqref{e:TransCond}
at $\xbar$, the set $B$ is an affine subspace, and $A$ is
$(\varepsilon,\delta)$-subregular relative to $A\cap B$ at $\xbar$.
  \item\label{t:Phan} \cite[Theorem 4.3]{Pha15}
$\{A,B\}$ satisfies the transversality condition \eqref{e:TransCond}
at $\xbar$, and $A$ and $B$ are
$(\varepsilon,\delta)$-sub\-regular at $\xbar$.
  \end{enumerate}
\end{proposition}
\section{Elemental set regularity}\label{Reg_of_sets}
Elemental (sub)regularity defined next provides a unifying framework for the other notions
of set regularity given in Definition \ref{d:set regularities}.
\begin{definition}[elemental regularity of sets]\label{d:set regularity}
Let $\emptyset \neq A\subset\Euclid$, $\bx\in \E$ and $(a,v)\in \gph{\ncone{A}}$.
\begin{enumerate}
   \item\label{d:geom subreg} $A$ is
   {\em elementally subregular of order $\sigma$ relative to $B\subset \E$ at $\xbar$ for $\paren{a,v}$
   with constant $\varepsilon$}
if there exists a neighborhood $U$ of $\xbar$
such that
\begin{equation}\label{e:geom subreg}
\ip{v - \paren{b-b_A}}{b_A - a}\leq
\varepsilon\norm{v-\paren{b-b_A}}^{1+\sigma}\norm{b_A-a}\;\;\mbox{for all}\;\; b\in B \cap U, b_{A}\in P_{A}(b).
\end{equation}

   \item\label{d:uni geom subreg} $A$
is
{\em uniformly} elementally subregular
of order $\sigma$ relative to $B\subset \E$ at $\xbar$ for
$\paren{a,v}$
if for any $\varepsilon>0$, there is a neighborhood $U$
of $\xbar$  such that \eqref{e:geom subreg} holds.

   \item\label{d:geom reg}
$A$ is
{\em elementally regular of order $\sigma$ at $\xbar$ for $\paren{a,v}$
with constant $\varepsilon$} if there exists a neighborhood $V$ of $v$ such that, for all
$u\in\ncone{A}(a)\cap V$, $A$ is elementally subregular of order $\sigma$ relative to $B:=A$
at $\xbar$ for $\paren{a,u}$ with constant $\varepsilon$.

   \item
$A$ is
{\em uniformly} elementally regular of order $\sigma$
at $\xbar$ for $\paren{a,v}$
if there exists a neighborhood $V$ of $v$ such that, for all $u\in\ncone{A}(a)\cap V$, $A$ is uniformly
elementally subregular of order $\sigma$ relative to $B:=A$ at $\xbar$ for $\paren{a,u}$.
\end{enumerate}
If $B=\{\xbar\}$ in \eqref{d:geom subreg} and \eqref{d:uni geom subreg}, then the respective qualifier,
``relative to'' is dropped.
If $\sigma=0$, then the respective qualifier, ``of order''
is dropped in the description of the properties.
The {\em modulus of elemental (sub)regularity} is the infimum over all
$\varepsilon$ for which \eqref{e:geom subreg} holds.
\end{definition}

In all properties in Definition \ref{d:set regularity},
$\xbar$ need not be in $B$ and $a$ need not be in $U$,
although these are the main cases of interest for us.
When $\sigma=0$, the properties are trivial for any constant $\varepsilon\ge 1$, so the only
case of interest is elemental (sub)regularity with constant $\varepsilon<1$.

\begin{example}\label{eg:regularity examples}$~$
\begin{enumerate}[(a)]
 \item\label{eg:cross}  A \emph{cross}:
 \[
    A=\mathbb{R}\times\{0\}\cup \{0\}\times\mathbb{R}.
 \]
This example is of particular interest for the study of {\em sparsity constrained optimization.}
$ A$ is elementally regular at any $\xbar\neq 0$, say $\|\xbar\|>\delta>0$, for all
$(a,v)\in \gph{\ncone{A}}$ where $a\in \Ball_{\delta}(\xbar)$
with constant $\varepsilon=0$ and neighborhood $\Ball_\delta(\xbar)$.
The set $ A$ is not elementally regular at $\xbar=0$ for any $(0,v)\in \gph{\ncone{A}}$
since $\ncone{A}(0)= A$.
However, $ A$ is elementally subregular at $\xbar=0$ for all $(a,v)\in\gph{\ncone{A}}$
with constant $\varepsilon=0$ and neighborhood $\E$ since all vectors $a\in A$ are orthogonal to $\ncone{A}(a)$. \\
%

\item\label{eg:circle} A circle:
 \[
  A=\set{ (x_1,x_2)\in\R^2}{ x_1^2+x_2^2=1}.
 \]
This example is of particular interest for the \emph{phase retrieval problem}.
The set $ A$ is uniformly elementally regular at any $\xbar\in  A$ for all $(\xbar,v)\in \gph{\ncone{A}}$.
Indeed, note first that for any $\xbar \in  A$, $\ncone{A}(\xbar)$ consists of the line passing through the origin and $\xbar$.
Now, for any $\varepsilon \in (0,1)$, we choose $\delta = \varepsilon$.
Then for any $x\in  A\cap \Ball_{\delta}(\xbar)$, it holds $\cos\angle(-\xbar,x-\xbar)\le \delta\le \varepsilon$.
Hence, for all $x\in  A\cap \Ball_{\delta}(\xbar)$ and $v\in \ncone{A}(\xbar)$,
 \[
   \ip{v}{x-\xbar} = \cos\angle(v,x-\xbar)\|v\|\|x-\xbar\| \leq \cos\angle(-\xbar,x-\xbar)\|v\|\|x-\xbar\| \le \varepsilon\|v\|\|x-\xbar\|.
 \]

%
\item\label{eg:Pie1r}
\begin{eqnarray*}
  A&=&\set{ (x_1,x_2)\in\R^2}{ x_1^2+ x_2^2\leq 1,\, x_1\leq |x_2|}\subset\R^2,\\
   B&=&\set{ (x_1,x_2)\in\R^2}{x_1^2+ x_2^2\leq 1, x_2 \le x_1 \le 2x_2}\subset\R^2.
\end{eqnarray*}
The set $ A$ is elementally subregular relative to $B$ at $\xbar = 0$ for all
$(a,v)\in \gph{\rncone{A}{B}}$ with constant $\varepsilon=0$ and neighborhood
$\E$ since for all $b\in B$, $b_A\in P_{A}(b)$ and $v\in \rncone{A}{B}(a)$, there holds
\[
 \ip{v-(b-b_A)}{b_A-a} = \ip{v}{b_A-a}-\ip{b-b_A}{b_A-a} = 0.
\]
The set $A$, however, is not elementally regular at $\xbar=0$ for any $(0,v)\in\gph{\ncone{A}}$
because by choosing $x=tv \in  A$ (where $0\neq v\in  A\cap \ncone{A}(0)$, $t\downarrow 0$),
we get $\ip{v}{x}=\|v\|\|x\|>0$.
%
\end{enumerate}
\xqed
\end{example}

\begin{proposition}\label{t:set regularity equivalences}
Let $A$, $A'$ and $B$ be closed nonempty subsets of $\E$.
\begin{enumerate}
\item\label{t:Hoelder char}
Let $ A\cap B\neq \emptyset$ and
suppose that there is a neighborhood $W$ of
$\xbar\in A\cap B$ and a constant $\varepsilon>0$ such that
for each
\begin{equation}\label{e:Hoelder normals}
(a,v)\in V\equiv \set{(b_A,u)\in \gph{\pncone{A}}}
{u=b-b_A,
       \begin{array}{c}\mbox{ for }b \in  B\cap W\\
		      \mbox{ and }b_A \in P_A(b)\cap W
       \end{array}},
\end{equation}
it holds that
\begin{equation}\label{e:xbar neighborhood}
\xbar\in\intr U(a,v) \mbox{ where } U(a,v)\equiv\Ball_{(1+\varepsilon^2)\|v\|}(a+v).
\end{equation}
Then, $A$ is $\sigma$-H\"older regular relative to $B$ at $\xbar$
with constant $c=\varepsilon^2$ and neighborhood $W$ of $\xbar$
if and only if $ A$ is elementally subregular of order $\sigma$ relative to
$ A\cap P^{-1}_B\paren{a+v}$ at $\xbar$
for each $(a,v)\in V$ with constant $\varepsilon=\sqrt{c}$ and the respective neighborhood
$U(a,v)$.


\item\label{t:L-eps-del-subreg char}
Let $B\subset A$.  The set $A$ is $(A',\varepsilon,\delta)$-subregular relative to $B$ at
$\xbar\in  A$ if and only if $ A$ is
elementally subregular relative to $B$ at $\xbar$ for all
$(a,v)\in \gph{\rpncone{ A}{ A'}}$ where $a\in \Ball_{\delta}(\xbar)$
with constant $\varepsilon$ and neighborhood $\Ball_{\delta}(\xbar)$.
Consequently, $(\varepsilon,\delta)$-subregularity implies 0-H\"older regularity.

\item\label{t:L-eps-del-regular equivalence}
The set $ A$ is $( A',\varepsilon,\delta)$-regular at $\xbar\in A$ if and only if $ A$ is
elementally subregular relative to $ A$ at $\xbar$ for all
$(a,v)\in \gph{\rpncone{ A}{ A'}}$ where $a\in \Ball_{\delta}(\xbar)$ with constant
$\varepsilon$ and neighborhood $\Ball_{\delta}(\xbar)$.
If $ A$ is $( A',\varepsilon,\delta)$-regular at $\xbar$, then $ A$ is elementally regular at $\xbar$
for all $(\xbar,v)$ with constant $\varepsilon$, where $0\neq v\in \rpncone{ A}{ A'}(\xbar)$.

\item\label{t:Clarke char}
The set $ A$ is Clarke regular at $\xbar\in  A$ if and only if $ A$ is
uniformly elementally regular at $\xbar$ for all $(\xbar,v)$ with $v\in\ncone{A}(\xbar)$.
Consequently, Clarke regularity implies $(\varepsilon,\delta)$-regularity.

\item\label{t:super-regular equivalence}
The set $ A$ is super-regular at $\xbar\in  A$
if and only if for any $\varepsilon>0$, there is a $\delta>0$
such that $A$ is elementally regular at $\xbar$ for all
$(a,v)\in \gph{\ncone{A}}$ where
$a\in \Ball_{\delta}(\xbar)$
with constant $\varepsilon$ and neighborhood $\Ball_{\delta}(\xbar)$.
Consequently, super-regularity implies Clarke regularity.

%
\item\label{t:prox-regular equivalence} If $ A$ is prox-regular at $\xbar$, then
there exist positive constants $\overline{\varepsilon}$ and $\deltabar$
such that, for any $\varepsilon>0$ and $\delta:=\frac{\varepsilon\deltabar}{\overline{\varepsilon}}$ defined correspondingly,
$ A$ is elementally regular at $\xbar$ for all
$(a,v)\in\gph{\ncone{A}}$ where $a\in \Ball_\delta(\xbar)$ with constant $\varepsilon$ and neighborhood $\Ball_\delta(\xbar)$.
Consequently, prox-regularity implies super-regularity.

\item\label{t: convex set regularity} If $ A$ is convex then it is
elementally regular at all $x \in A$ for all
$(a,v)\in \gph{\ncone{A}}$ with constant $\varepsilon=0$
and the neighborhood $\E$ for both $x$ and $v$.
\end{enumerate}
\end{proposition}

\begin{proof}
\eqref{t:Hoelder char}.
The set $A$ is  $\sigma$-H\"older regular at $\xbar $ relative to $B$ with constant
$c=\varepsilon^2$ and neighborhood $W$
if and only if
$$
\ip{b-b_A}{x-b_A}\le \varepsilon\|b-b_A\|^{1+\sigma}\|x-b_A\|
$$
holds for each $b\in B\cap W$, $b_A\in P_{A}(b)\cap W$, and
$x\in A\cap P_{B}^{-1}(b)\cap \Ball_{(1+\varepsilon^2)\|b-b_A\|}(b)$.
But this is equivalent to
$$
\ip{v}{x-a}\le \varepsilon\|v\|^{1+\sigma}\|x-a\|
$$
holding for each $(a,v)\in V$ and $x\in \paren{A\cap P_{B}^{-1}(a+v)}\cap U(a,v)$.
Thanks to assumption \eqref{e:xbar neighborhood}, this is equivalent to $ A$ being elementally
subregular of order $\sigma$ relative to $A\cap P_{B}^{-1}(a+v)$ at $\xbar$ for each
$(a,v)\in V$ with constant $\varepsilon$ and neighborhood $U(a,v)$.\\
%
%

\eqref{t:L-eps-del-subreg char}.
The set  $A$ is  $( A', \varepsilon, \delta)$-subregular relative to $B\subset A$ at $\xbar$
if and only if \eqref{eq:epsilondeltasubregularity} holds for all $b\in B\cap \Ball_{\delta}(\bx)$, $a\in A\cap \Ball_\delta(\xbar)$
and $v\in\rpncone{A}{A'}(a)$.
In other words, $A$ is elementally
subregular relative to $B$ at $\xbar$ for all $(a,v)\in \gph{\rpncone{ A}{ A'}}$ where $a\in \Ball_{\delta}(\xbar)$
with constant $\varepsilon$ and neighborhood $\Ball_\delta(\xbar)$. \\
%

\eqref{t:L-eps-del-regular equivalence}.
The first part is a particular case of \eqref{t:L-eps-del-subreg char} for $B=A$.
For the latter part, we suppose $ A$ is $(A',\varepsilon,\delta)$-regular at $\xbar$ and let
$0\neq v\in \rpncone{A}{A'}(\xbar)$.
%
%
We can assume without loss of generality that
$\xbar\in P_{A}(z)$ for $z\equiv\xbar+v$
(otherwise one could rescale $v$ so that this holds).
From the variational characterization of the Euclidean projector, $\xbar\in P_{ A}(z)$ if and only if
\[
 \ip{z-\xbar}{x-\xbar}\leq \frac{1}{2}\norm{x-\xbar}^2\;\; \mbox{for all}\;\; x\in A.
\]
In particular,
\[
  	\ip{v}{x-\xbar}\leq \frac{\varepsilon}{2}\norm{v}\norm{x-\xbar}\;\; \mbox{for all}\;\;
  	x\in  A\cap \Ball_{\deltabar}(\xbar)\quad (\mbox{where } \deltabar=\varepsilon\norm{v}>0).
\]
Consider the neighborhood $V:= \set{u\in \E}{(2+\varepsilon)\|u-v\|\le \varepsilon\|u\|}$ of $v$.
Then the claim holds since, for all $x\in A\cap \Ball_{\deltabar}(\xbar)$ and $u\in V$, one has
\begin{align*}
\ip{u}{x-\xbar} =\, & \ip{v}{x-\xbar} + \ip{u-v}{x-\xbar}\\
\le\, & \frac{\varepsilon}{2}\norm{v}\norm{x-\xbar} + \|u-v\|\|x-\xbar\|\\
\le\, & \frac{\varepsilon}{2}\norm{u}\norm{x-\xbar} + \frac{2+\varepsilon}{2}\|u-v\|\|x-\xbar\|\\
\le\, & \frac{\varepsilon}{2}\norm{u}\norm{x-\xbar} + \frac{\varepsilon}{2}\|u\|\|x-\xbar\|
= \varepsilon\norm{u}\norm{x-\xbar}.
\end{align*}
%

\eqref{t:Clarke char}. The set $ A$ is Clarke regular at $\xbar$ if and only if for any
$v\in\ncone{A}(\xbar)$ and $\varepsilon>0$, there is a $\delta>0$
such that
  \[
  \ip{v}{x-\xbar}\leq \varepsilon\|v\|\|x-\xbar\|\;\; \mbox{for all}\;\; x\in  A\cap\Ball_\delta(\xbar).
  \]
This means that $ A$ is uniformly elementally subregular relative to $ A$ at $\xbar$
for all $(\xbar,v)\in \gph{\ncone{A}}$.\\
%

\eqref{t:super-regular equivalence}. By \cite[Proposition 4.4]{LewLukMal09}, $ A$ is super-regular at $\xbar$
if and only if for any $\varepsilon>0$, there is a
$\delta>0$ such that
\begin{equation}\label{Th_ii}
\ip{v}{x-a}\leq \varepsilon\norm{v}\|x-a\|\;\; \mbox{for all}\;\; x,a\in  A\cap \Ball_\delta(\xbar), v\in \ncone{A}(a).
\end{equation}
Condition \eqref{Th_ii} just means that $A$ is elementally regular at $\xbar$ for all
$(a,v)\in \gph{\ncone{A}}$ where $a\in\Ball_{\delta}(\xbar)$
with constant $\varepsilon$ and neighborhood $\Ball_{\delta}(\xbar)$.
In particular, $A$ is uniformly elementally regular at $\xbar$ for all $(\xbar,v)$ with $v\in \ncone{A}(\xbar)$.
%
%
%
That is thanks to part \eqref{t:Clarke char}, $ A$ is Clarke regular at $\xbar$ confirming
the fact established in \cite[Corollary 4.5]{LewLukMal09} that super-regularity implies
Clarke regularity.\\
%

\eqref{t:prox-regular equivalence}. By \cite[Proposition 1.2]{PolRocThi00}, $ A$ is
prox-regular at $\xbar\in A$ if and only if $ A$ is prox-regular at $\xbar$ for $0$.
This means that there exist $\bar\varepsilon,\bar\delta>0$ such that
\begin{equation}\label{Th_Prox}
\ip{v}{x-a}\leq \tfrac{\bar\varepsilon}{2}\norm{x-a}^2\;\; \mbox{for all}\;\;
x,a\in  A\cap \Ball_{\bar\delta}(\xbar), v\in \ncone{A}(a)\cap (\bar{\delta}\Ball).
\end{equation}
Let $\varepsilon>0$ be arbitrary and define $\delta = \tfrac{\varepsilon\bar{\delta}}{\bar{\varepsilon}}>0$.
Then, for any $x,a\in  A\cap \Ball_{\delta}(\xbar)$ and $0\neq v\in \ncone{A}(a)$, condition \eqref{Th_Prox} ensures that
\begin{align*}
\ip{v}{x-a} =\, & \frac{\|v\|}{\bar{\delta}}\ip{\frac{\bar{\delta}}{\|v\|}v}{x-a} \leq \frac{\|v\|}{\bar{\delta}}\frac{\bar\varepsilon}{2}\norm{x-a}^2 \quad \paren{\mbox{as } \tfrac{\bar{\delta}}{\|v\|}v\in \bar{\delta}\Ball}\\
= &\, \frac{\bar{\varepsilon}\|x-a\|}{2\bar{\delta}}\|v\|\norm{x-a}
\leq \frac{\bar{\varepsilon}\delta}{\bar{\delta}}\|v\|\norm{x-a}\quad (\mbox{as } \|x-a\|\le 2\delta)\\
=\, &\varepsilon\|v\|\norm{x-a}.
\end{align*}
This just says that $ A$ is elementally regular at $\xbar$ for all
$(a,v)\in \gph{\ncone{A}}$ where $a\in\Ball_{\delta}(\xbar)$
with constant $\varepsilon$ and neighborhood $\Ball_{\delta}(\xbar)$.
Thanks to part \eqref{t:super-regular equivalence} this, in particular, shows that prox-regularity implies
super-regularity.\\
%
%

\eqref{t: convex set regularity}. Since we are in a finite dimensional setting,
$ A$ is nonempty closed and convex if and only if the projector is everywhere
single-valued ({\em Chebyshev} \cite[Theorem 3.14]{BauCom11}, \cite{Bun34},
\cite[Theorem 12.7]{Deu01}) and
\[
 \ip{z-a}{x-a}\leq 0\;\; \mbox{for all}\;\; x,a\in  A, z\in P^{-1}_ A(a).
\]
Since $ A$ is convex, then $\ncone{A}=\pncone{A}$,
and, in particular, $\set{z-a}{z\in P^{-1}_A(a)}= \ncone{A}(a)$ for any $a\in A$,
so the above inequality holds for all $x\in A$ and $(a,v)\in\gph{\ncone{A}}$.
That is, $ A$ is elementally regular at all $x\in A$ for all $(a,v)\in\gph{\ncone{A}}$
with constant $\varepsilon=0$ and the neighborhood $\E$ for both $x$ and $v$.\\
\qed
\end{proof}

\begin{remark}\label{Rem_1.3}
The convex characterization can be localized.   If there exists a $\rho>0$
such that $ A\cap \Ball_{\rho}(\xbar)$ is convex, then for any
$\delta\in (0,\rho)$ the set $ A$ is elementally regular at $\xbar$ for
all $(a,v)\in
\gph{\ncone{A}}$
where $a\in \Ball_{\delta}(\xbar)$
with constant $\varepsilon=0$ and neighborhood $\Ball_{\delta}(\xbar)$.
This is relevant for the set in Example \ref{eg:regularity examples}\eqref{eg:cross}.
\xqed
\end{remark}

\begin{remark} Assumption \eqref{e:xbar neighborhood} seems to be a technical one.
However, it is satisfied, for example, when the collection of sets $\{ A,  B\}$ is
strongly subtransversal at $\xbar$ with constant
$\alpha \ge \tfrac{1}{1+\varepsilon^2}$, where $\alpha$ is given below in Theorem \ref{t:MAINi}\eqref{t:MAINid},
since
\[
\|(a+v)-\xbar\|=\|b-\xbar\|=d(b, A\cap   B)\le (1+\varepsilon^2)d(b,A) = (1+\varepsilon^2)\|v\|.
\]
As a consequence, 0-H\"older regularity and elemental subregularity
as specified in Proposition \ref{t:set regularity equivalences} \eqref{t:Hoelder char} are equivalent
under the additional assumption of strong subtransversality of the collection of sets.
This observation falls within our interest of investigating relationships amongst various regularity
notions of individual sets and collections of sets.
\xqed
\end{remark}

\section{Regularity of collections of sets}\label{Reg_of_Col}
The following definition captures
two of the central notions found (under various aliases and disguises) in the literature.

\begin{definition}[transversality and subtransversality]\label{D2}
\begin{enumerate}
\item\label{D2i}
$\{A,B\}$ is \emph{subtransversal} at $\xbar $ if there exist numbers $\alpha>0$ and $\delta>0$ such that
\begin{equation}\label{D2-1}
\bigl(A+(\alpha\rho)\B\bigr)\cap \bigl(B+(\alpha\rho)\B\bigr)\cap\B_{\delta}(\xbar )\subseteq
\left(A\cap B\right)+\rho\B
\end{equation}
for all $\rho\in (0,\delta)$.

If, additionally, $\xbar $ is an isolated point of $A\cap B$, then $\{A,B\}$ is called
\emph{strongly subtransversal} at $\xbar $.
The (possibly infinite) supremum of all $\al$ above is denoted $\sr$
with the convention that the supremum of the empty set is zero.
\item\label{D2ii}
$\{A,B\}$ is \emph{transversal} at $\xbar $ if there exist numbers $\alpha>0$ and $\delta>0$ such that
\begin{equation}\label{D2-2}
(A-a-x_1)\cap(B-b-x_2)\cap(\rho\B )\neq \emptyset
\end{equation}
for all $\rho \in (0,\delta)$, $a\in A\cap\B_{\delta}(\xbar )$, $b\in B\cap\B_{\delta}(\xbar )$,
and all $x_1,x_2\in\E$ with $\max\{\|x_1\|,\|x_2\|\}<\alpha\rho$.
The (possibly infinite) supremum of all $\al$ above is denoted by $\rg$
with the convention that the supremum of the empty set is zero.
\end{enumerate}
\end{definition}
\begin{remark}\label{R1}
The maximum of the norms in Definition~\ref{D2} -- explicitly present in part \eqref{D2ii} and
implicitly also in part \eqref{D2i} --
 corresponds to the maximum norm in $\R^2$ employed in these definitions and subsequent assertions.
It can be replaced everywhere by the sum norm (pretty common in this type of definition in the literature) or any
other equivalent norm.
All the assertions that follow including the quantitative characterizations remain valid
(as long as the same norm is used everywhere), although the exact values $\sr$ and $\rg$ do depend on the chosen
norm and some estimates can change.  Note that the currently used maximum norm is not Euclidean.
These details become important in the context of applications where one norm may be more appropriate
than another.
\xqed
\end{remark}

Definition~\ref{D2}\eqref{D2i} was introduced recently in \cite{KruTha15} and can be viewed
as a local analogue of the global
\emph{uniform normal property} introduced in the convex setting in \cite[Definition~3.1(4)]{BakDeuLi05}
as a generalization of the \emph{property (N)} of convex cones by Jameson \cite{Jam72}.
A particular case of the Jameson property (N) for convex cones $A$ and $B$ such that $B=-A$ and $A\cap(-A)=\{0\}$
was studied by M. Krein in the 1940s.
Definition~\ref{D2}\eqref{D2ii} first appeared in \cite{Kru05} (see also \cite{Kru06,Kru09}) in the normed linear space
setting, where the property was referred to as simply \emph{regularity} (and later as \emph{strong regularity} and \emph{uniform regularity}).
In \cite{LewLukMal09}, the property is called \emph{linearly regular intersection}.

\begin{example}\label{E2}
If $\bx\in\Int(A\cap B)$, then $\{A,B\}$ is trivially transversal (and consequently subregular)
at $\xbar $ with any $\alpha>0$.
Thus, $\rg=\sr=\infty$.
\xqed\end{example}

\begin{example}\label{E3}
If $A=B\ne\E$, then $A+(\alpha\rho)\B= B+(\alpha\rho)\B$ and $A\cap B+\rho\B=A+\rho\B$.
Hence, condition \eqref{D2-1} holds (with any $\de>0$) if and only if $\al\le1$.
Thus, $\{A,B\}$ is subtransversal at $\xbar $ and $\sr=1$.
If $A=B=\E$, then we are in the setting of Example~\ref{E2}.
\xqed\end{example}

Note that, under the conditions of Example~\ref{E3}, $\{A,B\}$ does not have to be transversal at $\xbar $.

\begin{example}\label{E4}
Consider again the two sets in Example~\ref{E1}.
Thanks to Example~\ref{E3}, $\{A,B\}$ is subtransversal at $\xbar $ and $\sr=1$.
At the same time, $A-a=B-b=\R\times\{0\}$ for any $a\in A$ and $b\in B$.
If $x_1=(0,\eps)$ and $x_2=(0,0)$, then condition \eqref{D2-2} does not hold for any $\eps>0$ and $\rho>0$.
Thus, $\{A,B\}$ is not transversal at $\xbar $ and $\rg=0$.
\xqed\end{example}

The next two results are a catalog of the main characterizations of subtransversality and transversality, respectively.

\begin{theorem}[characterizations of subtransversality]\label{t:MAINi}
The following statements are equivalent to $\{A,B\}$ being subtransversal at $\xbar $.\\

\begin{enumerate}[\rm (i)]
\item\label{t:MAINib} There exist numbers
$\delta>0$ and $\alpha>0$ such that
\begin{equation*}
(A-x)\cap(B-x)\cap(\rho\B )\neq \emptyset
\end{equation*}
for all $x\in\B_{\delta}(\xbar )$ such that $x=a+x_1=b+x_2$ for some $a\in A$, $b\in B$ and $x_1,x_2\in\E$
with $\max\{\|x_1\|,\|x_2\|\}<\alpha\rho$.
Moreover, $\sr$ is the exact upper bound of all numbers $\alpha$ such that the condition above is satisfied.
\end{enumerate}
Metric characterizations.
\begin{enumerate}[\rm (i)]
\addtocounter{enumi}{1}
\item\label{t:MAINic} There exist numbers
$\delta>0$ and $\alpha>0$
such that
\begin{gather}\label{e:T1-1}
\alpha d\left(x,A\cap B\right)\le \max\left\{d(x,A),d(x,B)\right\}
\;\;\mbox{for all}\;\; x\in \B_{\delta}(\xbar ).
\end{gather}
Moreover, $\sr$ is the exact upper bound of all numbers $\alpha$ such that (\ref{e:T1-1}) is satisfied.
\item\label{t:MAINid}  There exist numbers
$\delta>0$ and $\alpha>0$
such that
\begin{equation}\label{e:P3-1}
\alpha d(x,A\cap B)\le d(x,B)\;\;\mbox{for all}\;\; x\in A\cap\B_{\delta}(\xbar ).
\end{equation}
Moreover,
\begin{equation}\label{P3-2}
\frac{\srr}{\srr+2}\le\sr\le\srr,
\end{equation}
where $\srr$ is the exact upper bound of all numbers $\alpha$ such that condition (\ref{e:P3-1}) is satisfied.
\end{enumerate}
\end{theorem}
\proof
{\em Characterization \eqref{t:MAINib}.}  This is easily checked. \\

{\em Characterization \eqref{t:MAINic}.}
This follows from \cite[Theorem~3.1]{KruTha15}.\\

{\em Characterization \eqref{t:MAINid}.}
The inequality \eqref{e:T1-1} implies \eqref{e:P3-1},
hence, thanks to characterization \eqref{t:MAINic}, subregularity of $\{A,B\}$ implies property \eqref{e:P3-1} with the
same numbers $\alpha>0$ and $\delta>0$, and the second inequality in \eqref{P3-2} holds true.

Conversely, let condition \eqref{e:P3-1} be satisfied with some numbers $\alpha>0$ and $\delta>0$.
Choose a positive number $\al'<\al/(\al+2)$ and a number $\be\in(\al/2,\al)$ such that
$\al'<1/(\al/\be+\al\iv+\be\iv)$ which is always possible if $\be$ is taken sufficiently close to $\al$.
For any $x'\in\B_{\frac{\delta}{3}}(\bar{x})$, we choose an $x\in A$ such that $\frac{\be}{\al}\norm{x'-x}\le d(x',A)$.
Note that $x\in\B_{\delta}(\bar{x})$ since
\begin{equation*}
\norm{x-\bar{x}}\le \norm{x-x'}+\norm{x'-\bar{x}}\le \frac{\al}{\be}d(x',A)+\norm{x'-\bar{x}}\le
\left(\frac{\al}{\be}+1\right)\norm{x'-\bar{x}}\le 3\norm{x'-\bar{x}}\le \delta.
\end{equation*}
Using \eqref{e:P3-1}, we obtain
\begin{align*}
d(x',A\cap B) &\le\norm{x'-x}+d(x,A\cap B)\\
&\le \norm{x'-x} +\al\iv d(x,B)\\
&\le\norm{x'-x} +\al\iv\left(\norm{x-x'}+d(x',B)\right)\\
&=\left(1+\al\iv\right)\norm{x'-x} +\al\iv d(x',B)\\
&\le\left(1+\al\iv\right)\frac{\al}{\be}d(x',A) +\al\iv d(x',B)\\
&\le \left(\frac{\al}{\be}+\be\iv+\al\iv\right) \max\{d(x',A),d(x',B)\}\\
&\le(\al')\iv\max\{d(x',A),d(x',B)\},
\end{align*}
and consequently, after passing to the limit as $\al'\to\al/(\al+2)$,
\begin{align*}
\frac{\al}{\al+2}d(x',A\cap B) \le\max\{d(x',A),d(x',B)\}.
\end{align*}
Hence, thanks to characterization \eqref{t:MAINic}, property \eqref{e:P3-1} implies subtransversality of $\{A,B\}$ with
numbers $\al/(\al+2)$ and $\delta/3$.
Because $\al$ can be chosen arbitrarily close to $\srr$, the first inequality in \eqref{P3-2} holds true.\\
This completes the proof.
\qed
\bigskip

\begin{remark}[Historical remarks and further relations]
Thanks to characterization \eqref{t:MAINic} of Theorem~\ref{t:MAINi}, subtransversality of a collection of sets can be
recognized as a well known regularity property that has been around for more than 20 years under
the names of (local) \emph{linear regularity}, \emph{metric regularity}, \emph{linear coherence},
\emph{metric inequality}, and \emph{subtransversality}; cf. \cite{BakDeuLi05,BauBor93,
BauBor96,Iof89,Iof00,Iof15,KlaLi99,HesLuk13,
LiNgPon07,NgaThe01,Pen13,ZheNg08,ZheWeiYao10, DruIofLew15,Pan15}.
It has been used as the key assumption when establishing linear convergence of sequences generated by
cyclic projection algorithms and a qualification condition for subdifferential and normal cone calculus formulae.
This property is implied by the \emph{bounded linearly regularity} \cite{BauBor96}.
If $A$ and $B$ are closed convex sets and the collection $\{A,B\}$ is subtransversal at any point in $A\cap B$,
then it is boundedly linear regular; cf. \cite[Remark~6.1(d)]{BakDeuLi05}.
Characterization \eqref{t:MAINid} of Theorem \ref{t:MAINi}
can be considered as a nonconvex extension of \cite[Theorem~3.1]{NgYan04}.

One can also observe that condition \eqref{e:T1-1} is equivalent to the function $x\mapsto\max\{d(x,A),d(x,B)\}$
having a local \emph{error bound} \cite{Aze03,FabHenKruOut10,Kru15}/\emph{weak sharp minimum}
\cite{BurDen02,BurDen05,BurFer93} at $\bx$ with constant $\al$.
One can think of condition \eqref{e:T1-2} as a kind of uniform local error bound/relaxed weak sharp minimum
property; cf. \cite{Kru06}.

The geometrical property \eqref{e:T1-1} of a collection of sets $\{A,B\}$ can also be viewed as a
certain property of a collection of distance functions $x\mapsto d(x,A)$ and $x\mapsto d(x,B)$.
It is sufficient to notice that $A\cap B=\left\{x\in \Euclid \mid  \max\{d(x,A),d(x,B)\}\le 0\right\}$.
One can study regularity properties of collections of arbitrary (not necessarily distance) functions.
Such an attempt has been made recently in the convex setting by C.H.J.~Pang \cite{Pan15}.
Given a collection of convex functions $\{f_1,f_2\}$, the following analogue of condition \eqref{e:T1-1} is
considered in \cite{Pan15}:
\begin{gather*}
\alpha d\left(x,C\right)\le \max\left\{d(x,H_1(x)),d(x,H_2(x))\right\}
\;\;\mbox{for all}\;\; x\in \E,
\end{gather*}
where $C:=\{u\in\E\mid\max\{f_1(u),f_2(u)\}\le0\}$, $H_i(x):=\{u\in\E\mid f_i(x)+\ang{v_i,u-x}\le0\}$ for
some chosen $v_i\in\sd f_i(x)$ if $f_i(x)>0$ and $H_i(x):=\E$ otherwise, $i=1,2$.
It is easy to check that, in the case of distance functions, this property reduces to \eqref{e:T1-1}.
\xqed
\end{remark}

\begin{theorem}[characterizations of transversality]\label{t:MAINii}
The following statements are equivalent to $\{A,B\}$ being
transversal at $\xbar$.\\

\noindent Metric characterizations.
\begin{enumerate}[\rm (i)]
\item\label{t:MAINiib} There exist numbers $\delta>0$ and $\alpha>0$
such that
\begin{gather}\label{e:T1-2}
\alpha d\left(x,(A-x_1)\cap (B-x_2)\right)\le \max\left\{d(x,A-x_1),d(x,B-x_2)\right\},\;\;
\forall x\in \B_{\delta}(\xbar ),\;x_1,x_2\in \delta\B.
\end{gather}
Moreover, $\rg$ is the exact upper bound of all numbers $\alpha$ such that (\ref{e:T1-2}) is satisfied.
\item\label{t:MAINiic} There exist numbers $\delta>0$ and $\alpha>0$
such that
\begin{equation}\label{e:P3-3}
\alpha d(x,(A-x_1)\cap(B-x_2))\le d(x,B-x_2),\;\;
\forall x\in (A-x_1)\cap\B_{\delta}(\xbar ),\;x_1,x_2\in\de\B.
\end{equation}
Moreover,
\begin{equation}\label{P3-4}
\frac{\rgg}{\rgg+2}\le\rg\le\rgg,
\end{equation}
where $\rgg$ is the exact upper bound of all numbers $\alpha$ such that condition (\ref{e:P3-3}) is satisfied.
\end{enumerate}
Dual characterizations.
\begin{enumerate}[\rm (i)]
\addtocounter{enumi}{2}
\item\label{t:MAINiie}
There exist numbers $\alpha>0$
and $\delta>0$ such that
$\|v_1+v_2\|>\alpha$
for all $a\in A\cap\B_\delta(\bar{x})$, $b\in B\cap\B_\delta(\bar{x})$,
$v_1\in\fncone{A}(a)$ and $v_2\in\fncone{B}(b)$ with $\|v_1\|+\|v_2\|=1$.
Moreover, $\rg$ is the exact upper bound of all numbers $\alpha$ above.
\item\label{t:MAINiie-n} There exists a number $\alpha>0$ such that
$\|v_1+v_2\|>\alpha$
for all
$v_1\in\ncone{A}(a)$ and $v_2\in\ncone{B}(b)$ with $\|v_1\|+\|v_2\|=1$.
Moreover, $\rg$ is the exact upper bound of all such numbers $\alpha$.\\

\item\label{t:MAINiid}
$\ncone{A}(\xbar ) \cap (-\ncone{B}(\xbar )) = \{0\}$.\\

\item\label{t:MAINiig}
There is a number $\alpha>0$ such that
$d^2\left(v,\ncone{A}(\bar{x})\right)+ d^2\left(v,-\ncone{B}(\bar{x})\right)>\alpha^2$
for all $v\in\Sp$.
Moreover, the exact upper bound of all such numbers $\alpha$, denoted $\rgdd$,
satisfies $\rgdd=\sqrt{2}\,\rg$.

\item\label{t:MAINiij}
There is a number $\alpha<1$ such that
$-\ang{v_1,v_2}<\alpha$
for all $v_1\in\ncone{A}(\bx)$ and $v_2\in\ncone{B}(\bx)$ with $\|v_1\|=\|v_2\|=1$.
Moreover, the exact lower bound of all such numbers $\alpha$, denoted $\rga$,
satisfies $\rga+2(\rg)^2=1$.
\end{enumerate}
\end{theorem}
\proof
{\em Characterization \eqref{t:MAINiib}.}
  This follows from \cite[Theorem~3.1]{KruTha15}.\\

{\em Characterization \eqref{t:MAINiic}.}
Condition \eqref{e:T1-2} implies \eqref{e:P3-3}.
Hence, 
transversality  of $\{A,B\}$
implies property \eqref{e:P3-3} with the same numbers $\alpha>0$ and $\delta>0$,
and the second inequality in \eqref{P3-4} holds true.

Conversely, let condition \eqref{e:P3-3} be satisfied with some numbers $\alpha>0$ and $\delta>0$.
Fix any $x_1,x_2\in\de'\B$.
Then condition \eqref{e:P3-1} is satisfied with the same numbers $\alpha$ and $\delta$ and the
sets $A-x_1$ and $B-x_2$ in place of $A$ and $B$, respectively.
As shown in the proof of Theorem \ref{t:MAINi} \eqref{t:MAINid}, it holds
\begin{gather*}
\frac{\al}{\al+2}d\left(x',(A-x_1)\cap (B-x_2)\right)\le \max\left\{d(x',A-x_1),d(x',B-x_2)\right\}
\;\;\mbox{for all}\;\; x'\in \B_{\frac{\delta}{3}}(\bar{x}).
\end{gather*}
Since the points $x_1,x_2\in\de\B$ are arbitrary, this is exactly the transversality  property of $\{A,B\}$
with numbers $\al/(\al+2)$ and $\delta/3$.
The first inequality in \eqref{P3-4} follows.\\

{\em Characterization \eqref{t:MAINiie}.}  This was established in \cite{Kru05} (in the Asplund space setting).\\

{\em Characterization \eqref{t:MAINiie-n}.}
Thanks to the representation \eqref{NC4}, taking limits in characterization \eqref{t:MAINiie} yields
the result. \\

{\em Characterization \eqref{t:MAINiid}.}
The equivalence of conditions \eqref{t:MAINiie-n} and \eqref{t:MAINiid}
is straightforward since $\ncone{A}(\bar{x})$ and $\ncone{B}(\bar{x})$ are closed cones.\\

{\em Characterization \eqref{t:MAINiij}.}
This has been established in \cite[Theorem~2]{KruTha13}).\\

{\em Characterization \eqref{t:MAINiig}.}
Let $v\in\Sp$ be such that
$d^2\left(v,\ncone{A}(\bar{x})\right)+ d^2\left(v,-\ncone{B}(\bar{x})\right)=(\rgdd)^2$, and
$v_1$ and $v_2$ be the projections of $v$ on $\ncone{A}(\bx)$ and $-\ncone{B}(\bx)$,
respectively.
It is immediate from the symmetry that $\|v_1\|=\|v_2\|$.
Thus,
$$
\ang{v,v_1+v_2}=\|v_1+v_2\|,
\quad
\ang{v,v_1}=\ang{v,v_2}=\|v_1\|^2=\|v_2\|^2,
$$
and consequently,
\begin{gather*}
\|v_1+v_2\|=\|v_1\|^2+\|v_2\|^2=2\|v_1\|^2,
\\
\|v_1-v_2\|^2=2\|v_1\|^2-2\ang{v_1,v_2}= 4\|v_1\|^2-\|v_1+v_2\|^2= 4\|v_1\|^2(1-\|v_1\|^2).
\end{gather*}
Let $v_1'$ and $v_2'$ be unit vectors such that $v_1=\|v_1\|v_1'$ and $v_2=-\|v_1\|v_2'$.
Then $v_1'\in\ncone{A}(\bx)\cap\Sp$, $v_2'\in\ncone{B}(\bx)\cap\Sp$, and
$
\|v_1'+v_2'\|^2=4(1-\|v_1\|^2).
$
At the same time,
\begin{align*}
(\rgdd)^2&=\|v-v_1\|^2+\|v-v_2\|^2= 2-2\ang{v,v_1+v_2}+2\|v_1\|^2
\\&
=2(1-\|v_1\|^2)= \frac{1}{2}\|v_1'+v_2'\|^2\ge2(\rg)^2.
\end{align*}
Conversely, let $v_1\in\ncone{A}(\bx)$ and $v_2\in\ncone{B}(\bx)$ be such that
$\|v_1\|=\|v_2\|=1$ and $\|v_1+v_2\|=2\rg$.
Choose a unit vector $v$ such that $\ang{v,v_1+v_2}=\|v_1+v_2\|$ and let
$v_1'$ and $v_2'$ be the projections of $v$ on $v_1$ and $v_2$, respectively.
We are in a situation as above.
\begin{align*}
4(\rg)^2&=\|v_1'+v_2'\|^2= 2(\|v-v_1'\|^2+\|v-v_2'\|^2)\ge2(\rgdd)^2.
\end{align*}
This completes the proof of \eqref{t:MAINiig}.
%
\qed

\begin{remark}[characterization \eqref{t:MAINiib} and strong regularities]
Characterization \eqref{t:MAINiib} of Theorem~\ref{t:MAINii} reveals that the transversality of a collection
of sets corresponds to
subtransversality/linear regularity of all their small translations holding uniformly (cf. \cite[Page~1638]{DruIofLew15}).
Property \eqref{e:T1-2} was referred to in \cite{Kru05,Kru06,Kru09} as \emph{strong metric inequality}.
If $A$ and $B$ are closed convex sets and $\Int A\ne\emptyset$, then the transversality of the
collection $\{A,B\}$ is equivalent to the conventional qualification condition: $\Int A\cap B\ne\emptyset$; cf.
\cite[Proposition~14]{Kru05}.
\xqed
\end{remark}

\begin{remark}[characterization \eqref{t:MAINiie} and Jameson's property]
Characterization \eqref{t:MAINiie} in Theorem~\ref{t:MAINii}  can be formulated equivalently as follows:
\smallskip

\noindent\it
there exist numbers
$\alpha>0$ and $\delta>0$ such that
$\|v_1+v_2\|\ge\alpha(\|v_1\|+\|v_2\|)$
for all $a\in A\cap\B_\delta(\bar{x})$, $b\in B\cap\B_\delta(\bar{x})$,
$v_1\in\fncone{A}(a)$ and $v_2\in\fncone{B}(b)$.
\smallskip

\noindent\rm
This characterization can be interpreted as a strengthened version of the \emph{Jameson’s property (G)}
\cite{Jam72} (cf. \cite{NgYan04,BakDeuLi05,BauBorLi99}).
As with all dual characterizations, it basically requires that among all admissible pairs of nonzero normals
to the sets there is no pair of normals which are
oppositely directed.  Thanks to the representation \eqref{NC3}, one can replace the Fr\'echet normals
by the proximal ones.
\xqed
\end{remark}

\begin{remark}[characterization \eqref{t:MAINiid} and related notions]
Note that, unlike the other characterizations, \eqref{t:MAINiid} provides only a
qualitative criterion of transversality.
It has the interpretation that the cones
$\ncone{A}(\xbar )$ and $\ncone{B}(\xbar )$
are \emph{strongly additively regular} \cite{BurDen05}, and has been described as a
``concise, fundamental, and widely studied geometric property'' \cite{DruIofLew15} extensively used in
nonconvex optimization and calculus.
It is known under various names, e.g., \emph{transversality} \cite{ClaLedSteWol98},
\emph{normal qualification condition} \cite{Mor06,Pen13}, \emph{linearly regular intersection}
\cite{LewLukMal09}, \emph{alliedness property} \cite{Pen13}, and \emph{transversal intersection}
\cite{Iof15,DruIofLew15}.

An immediate consequence of characterization \eqref{t:MAINiid} is the following crucial inclusion
expressed in terms of the limiting normal cones
(cf. \cite[P.~99]{ClaLedSteWol98}, \cite[Theorem~6.42]{RocWet98}, \cite[P.~142]{Mor06}):
\begin{align}\label{cul}
\ncone{A\cap B}(\xbar )\subseteq \ncone{A}(\xbar)+\ncone{B}(\xbar),
\end{align}
which can be considered as an extension of the \emph{strong conical hull intersection property
(strong CHIP)} \cite{DeuLiWar97 } (cf. \cite[Definition~5.1(2)]{BakDeuLi05}) to nonconvex sets.
Indeed, since the opposite inclusion in terms of Fr\'echet normal cones holds true trivially:
\begin{align}\label{cul1}
\fncone{A\cap B}(\xbar )\supset \fncone{A}(\xbar)+\fncone{B}(\xbar),
\end{align}
and both cones reduce in the convex case to the normal cone \eqref{CNC}, inclusion \eqref{cul} is
equivalent in the convex setting to the strong CHIP:
\begin{align}\label{cul2}
\nsub{A\cap B}(\xbar )= \nsub{A}(\xbar)+\nsub{B}(\xbar).
\end{align}
The last equality has proved to be a fundamental regularity property in several areas of convex optimization;
see the discussion of the role of this property (and many other regularity properties of collections of convex
sets) in \cite{BakDeuLi05,BauBorLi99}.
Inclusion \eqref{cul} plays a similar role in nonconvex optimization and calculus.
Thus, thanks to Theorem \ref{t:MAINii}\eqref{t:MAINiid}, transversality implies
the extended strong CHIP \eqref{cul}.

In fact, it is now well recognized that inclusion \eqref{cul} is ensured by the weaker subtransversality property.
The next proposition is a consequence of \cite[Proposition~3.2]{IofOut08} (or \cite[Theorem~6.41]{Pen13})
and the characterization of subtransversality in
Theorem~\ref{t:MAINi}\eqref{t:MAINic}.
\begin{proposition}\label{P7}
If $\{A,B\}$ is subtransversal at $\xbar $, then inclusion \eqref{cul} holds true.
\end{proposition}
\noindent In the convex case, a nonlocal version of Proposition~\ref{P7} together with certain quantitative estimates
can be found in \cite{BakDeuLi05,BauBorLi99}.

If a stronger than \eqref{cul} condition \eqref{cul2} is satisfied in the nonconvex case (with Fr\'echet subdifferentials),
then this property is referred to in \cite{NgZan07} as the \emph{strong Fr\'echet-CHIP}.
Since inclusion \eqref{cul1} always holds, this is equivalent to inclusion \eqref{cul} with Fr\'echet
subdifferentials in place of the limiting ones.
A quantitative (by a positive number $\al$) version of the strong Fr\'echet-CHIP property was studied in the
convex and nonconvex settings in \cite{NgZan07,ZheNg08}:
$$
\fncone{A\cap B}(\bx) \cap \mathbb{B} \subseteq \alpha\paren{\paren{\fncone{A}(\bx) \cap
\mathbb{B}}+\paren{\fncone{B}(\bx) \cap \mathbb{B}}}.
$$
A number of important links with other regularity properties were established there, and
variants of the above property involving Clarke normal cones were also considered.

The ``restricted'' analogue of the dual regularity characterization
\eqref{t:MAINiid} takes the following form:
\begin{equation}\label{CQ}
\rncone{A}{B}(\xbar )\cap \left(-\rncone{B}{A}(\xbar )\right) \subseteq\{0\}.
\end{equation}
This condition is, incidentally, equivalent to the constraint qualification characterization of
regularity of collections of sets given in Definition \ref{D4} \eqref{D4:CQ}.
\xqed
\end{remark}

\begin{remark}[characterizations restricted to Euclidean settings]\label{r:KruTha13}
The following equivalent characterizations of transversality have been established in
\cite[Theorem~2]{KruTha13}).
\begin{enumerate}[(i)]
 \item\label{t:MAINiie-z}
There exists a number $\alpha>0$ such that
$\|v_1+v_2\|>2\alpha$
for all
$v_1\in\ncone{A}(\bx)$ and $v_2\in\ncone{B}(\bx)$ with $\|v_1\|=\|v_2\|=1$.
Moreover, the exact upper bound of all such numbers $\alpha$ equals $\rg$.

\item\label{t:MAINiif}
There exists a number $\alpha<1$ such that
$\|v_1-v_2\|<2\alpha$
for all
$v_1\in\ncone{A}(\bx)$ and $v_2\in\ncone{B}(\bx)$ with $\|v_1\|=\|v_2\|=1$.
Moreover, the exact lower bound of all such numbers $\alpha$, denoted $\rgd$,
satisfies $(\rg)^2+(\rgd)^2=1$.
\end{enumerate}
For brevity, the characterizations above  are in terms of limiting normals only.  The corresponding (approximate)
statements in terms of Fr\'echet and proximal normals can be formulated in a similar way.
These characterizations as well as that of Theorem \ref{t:MAINii}\eqref{t:MAINiig}
for the proximal normal cone only hold in Euclidean spaces.
\xqed
\end{remark}

\begin{remark}
Theorem \ref{t:MAINii}\eqref{t:MAINiij} also has analogues in terms of Fr\'echet and proximal normals.
The expression $-\ang{v_1,v_2}$  can be interpreted as the
cosine of the angle between the vectors $v_1$ and $-v_2$.
Note that, unlike $\rg$, $\rgd$, and $\rgdd$, constant $\rga$ can be negative.
Constant $\rga$ is a modification of another one:
\begin{align*}
\bar c:=\max&\left\{-\ang{v_1,v_2}\mid v_1\in\overline{N}_{A}(\bx)\cap\B,\; v_2\in\overline{N}_{B}(\bx)\cap\B\right\},
\end{align*}
used in \cite{LewLukMal09} for characterizing transversality.
It is easy to check that $\bar c=(\rga)_+$,
and $\bar c<1$ if and only if $\rga<1$.
\xqed
\end{remark}

Because the representation of (sub)transversality via set-valued mappings occupies
a rather special relation to these properties, we detail these separately next.
\begin{theorem}[characterizations via set-valued mappings]\label{t:svm}$~$

Subtransversality of the collection $\{A, B\}$ at a point $\xbar\in A\cap B$ has
the following eqvuialent characterizations.
\begin{enumerate}[\rm (i)]
\item\label{theorem13i} The set-valued mapping $F:\Euclid \rightrightarrows {\Euclid}^2$ given by \eqref{FAB}
with the max norm on $\Euclid^2$
 is metrically subregular at $\xbar $ for $(0,0)$. Moreover,
 \[
 \sr={\rm sr}[F](\bx,(0,0)).
 \]
 The mapping $F$ is strongly metrically subregular at $\xbar$ for $(0,0)$ if and only if
 the collection $\{A, B\}$ is strongly subtransversal there.
\item\label{NewThe1i} The set-valued mapping $G:\Euclid^2 \rightrightarrows {\Euclid}$ given by \eqref{F3}
with the Euclidean norm on $\Euclid^2$
is metrically subregular at $(\xbar ,\xbar )$ for $0$.
Moreover,
\begin{equation}\label{F3Mod}
\sqrt{\frac{2}{1+(\sr)^{-2}}}\le{\rm sr}[G]((\bx ,\bx ),0) \le\frac{2}{[(\sr)\iv-1]_+}.
\end{equation}
%
\end{enumerate}

Transversality of the collection $\{A, B\}$ at a point $\xbar\in A\cap B$ has
the following equivalent characterizations.
\begin{enumerate}[\rm (i$^{\prime}$)]
\item\label{theorem13ii} The set-valued mapping $F:\Euclid \rightrightarrows {\Euclid}^2$
given by \eqref{FAB}
with the max norm on $\Euclid^2$
is metrically regular at $\xbar$ for $(0,0)$.
Moreover,
\[
\rg={\rm r}[F](\bx,(0,0)).
\]
\item\label{NewThe1ii} The set-valued mapping $G:\Euclid^2 \rightrightarrows {\Euclid}$ given by \eqref{F3}
with the Euclidean norm on $\Euclid^2$
is metrically regular at $(\xbar ,\xbar )$ for $0$.
Moreover,
\begin{equation}\label{F3Mod2}
\sqrt{\frac{2}{1+(\rg)^{-2}}}\le{\rm r}[G]((\bx ,\bx ),0) \le\frac{2}{[(\rg)\iv-1]_+}.
\end{equation}
%
\end{enumerate}
\end{theorem}
\proof
{\em Characterization \eqref{theorem13i}}.
This is a consequence of \cite[Proposition~3.5]{Iof00}, \cite[Proposition 8]{Kru06} and
the characterization of Theorem~\ref{t:MAINi}\eqref{t:MAINic}.\\

{\em Characterization \eqref{NewThe1i}}.
Suppose $\{A,B\}$ is subtransversal at $\xbar$.
By Theorem~\ref{t:MAINi}\eqref{t:MAINic},
there exist numbers $\al>0$ and $\delta>0$ such that condition \eqref{e:T1-1} holds true.
Set $\al':=\sqrt{\frac{2}{1+\al^{-2}}}$.
We show that
\begin{equation}\label{F3MetSub}
\al' d((x_1,x_2), G^{-1}(0)) \le d(0, G(x_1,x_2)\}\qdtx{for all} x_1,x_2\in B_{\delta}(\xbar ).
\end{equation}
If $(x_1,x_2)\notin A\times B$, then $ G(x_1,x_2)=\emptyset$ and the inequality holds trivially.
Take any $(x_1,x_2)\in A\times B$ with $x_1,x_2\in B_{\delta}(\xbar )$.
Note that $ G(x_1,x_2)=x_1-x_2$ and $ G^{-1}(0) = \{(x,x)\mid  x\in A\cap B\}$.
Set $\hat x:=\frac{x_1+x_2}{2}$.
Then $\hat x\in B_{\delta}(\xbar )$, $x_1-\hat x=\hat x-x_2=\frac{x_1-x_2}{2}$, and, thanks to \eqref{e:T1-1},
\begin{equation}\label{Th3_Est1}
\alpha d(\hat x,A\cap B) \le\max\left\{d(\hat x,A),d(\hat x,B)\right\}
\le \max\left\{\norm{\hat x-z},\norm{\hat x-x_2}\right\}
=\frac{1}{2}\norm{x_1-x_2}.
\end{equation}
For any $x\in A\cap B$, we have
\begin{align*}
\|(x_1,x_2)-(x,x)\|^2 & = \norm{x_1-x}^2+\norm{x_2-x}^2=  \norm{x_1-\hat x+\hat x-x}^2+\norm{x_2-\hat x+\hat x-x}^2\\ &
= \norm{x_1-\hat x}^2+2\ang{x_1-\hat x,\hat x-x}+ \norm{x_2-\hat x}^2+2\ang{x_2-\hat x,\hat x-x}+ 2\norm{\hat x-x}^2\\
&= \norm{x_1-\hat x}^2+\norm{x_2-\hat x}^2+2\norm{\hat x-x}^2= \frac{1}{2}\norm{x_1-x_2}^2 + 2\norm{\hat x-x}^2.
\end{align*}
Hence,
\begin{align*}
d^2((x_1,x_2), G^{-1}(0))\le \frac{1}{2}\norm{x_1-x_2}^2 + 2d^2(\hat x,A\cap B),
\end{align*}
and, thanks to \eqref{Th3_Est1},
\begin{align}\label{e:NewThe1iq1}
\al'd((x_1,x_2), G^{-1}(0))&\le \al'\sqrt{\frac{1}{2}\left(1+\frac{1}{\alpha^2}\right)} \norm{x_1-x_2}=d(0, G(x_1,x_2)).
\end{align}

Conversely, suppose that $ G$ is metrically subregular at $(\xbar ,\xbar )$ for $0$, i.e.,
\eqref{F3MetSub} is satisfied for some numbers $\alpha'>0$ and $\delta>0$.
Fix an arbitrary number $\alpha\in (\alpha'/2,\alpha')$.
For any $z\in A\cap\B_{\frac{\delta}{3}}(\xbar )$, we pick a $w\in B$ such that $\norm{x_1-x_2}\le \frac{\alpha'}{\alpha}d(z,B)$.
Note that $w\in\B_{\delta}(\xbar )$ since
$$
\norm{x_2-\xbar }\le \norm{x_2-z}+\norm{x_1-\xbar }\le \frac{\alpha'}{\alpha}d(z,B)+\norm{x_1-\xbar }\le
\left(\frac{\alpha'}{\alpha}+1\right)\norm{x_1-\xbar }< \left(\frac{\alpha'}{\alpha}+1\right)\frac{\delta}{3}<\delta.
$$
Then, in view of \eqref{F3MetSub}, we have
\begin{align}\notag
\al d(z,A\cap B)&\le\al d((x_1,x_2),\{(x,x)\mid  x\in A\cap B\})=\al d((x_1,x_2), G^{-1}(0))
\\\label{The3_Est2}
&\le\frac{\alpha}{\alpha'}d(0, G(x_1,x_2))
=\frac{\alpha}{\alpha'}\norm{x_1-x_2} \le d(z,B).
\end{align}
By 
Theorem~\ref{t:MAINi}\eqref{t:MAINid}, condition \eqref{The3_Est2}
implies the subtransversality of $\{A,B\}$
at $\bar{x}$ and the estimate $\sr \ge 1/(1+2\al\iv)$, or equivalently, $\al\le2/[(\sr)\iv-1]_+$.
Passing to the limit in the last inequality as $\al\to\al'$ and then as
$\al'\to{\rm sr}[G]((\bar{x},\bar{x}),0)$, we arrive at the second inequality in \eqref{F3Mod}.

The equivalence of the strong subtransversality of $\{A,B\}$
and the strong metric subregularity of $G$ is straightforward since $x$ is an
isolated point of $\{A,B\}$ if and only if $(x,x)$ is an isolated point of  $G^{-1}(0)$.\\

{\em Characterization \eqref{theorem13ii}}.
This is a consequence of \cite[Proposition~3.5]{Iof00}, \cite[Proposition 8]{Kru06} and
the characterization
Theorem~\ref{t:MAINii}\eqref{t:MAINiib}.\\

{\em Characterization \eqref{NewThe1ii}}.
This is a consequence of Theorem \ref {t:MAINi}\eqref{NewThe1i} as it claims the
equivalence of the uniform versions of the properties in Theorem \ref {t:MAINi}\eqref{NewThe1i}.
The estimates established in Theorem \ref {t:MAINi}\eqref{NewThe1i} are preserved; cf. the proof of
Theorem~\ref{t:MAINii}\eqref{t:MAINiic}.
%
%
%
\qed

\begin{remark}[collections of sets and set-valued mappings]
The characterizations in Theorem \ref{t:svm}  provide a one-to-one
correspondence between regularity
properties of collections of sets and the corresponding ones of set-valued mappings.
They remain true for arbitrarily finite collections of sets.

The `positive part' sign in the \RHS s of the conditions \eqref{F3Mod} and \eqref{F3Mod2} is
used to accommodate for the case when $\bx\in\Int A\cap\Int B$ and, hence, $\rg=+\infty$.
In this case, conditions \eqref{F3Mod} and \eqref{F3Mod2} impose no upper bound on the
values ${\rm sr}[G]((\xbar ,\xbar ),0)$ and ${\rm r}[G]((\xbar ,\xbar ),0)$.
Recall that in the current article the assumption $\bx\in\cl A\cap\cl B$ is in forth, so the
`positive part' sign in the \RHS s of the conditions \eqref{F3Mod} and \eqref{F3Mod2} can be dropped.

In view of characterization \eqref{NewThe1ii}, the property of
\emph{regular intersection} of sets considered
in \cite[Section~5]{LewMal08} is equivalent to their collection being transversal.
This fact also follows from \cite[Theorem~5.1]{LewMal08}.
The regularity estimate obtained in parts \eqref{NewThe1i} and
\eqref{NewThe1ii} coincides with $\rgdd$.

Thanks to characterizations of Theorem \ref {t:svm},
when investigating regularity
properties of collections of sets one can employ the well developed regularity
theory of \SVM s, particularly, the celebrated
\emph{coderivative criterion} for metric regularity \cite{Mor06,RocWet98,DonRoc09} (see also
\cite{Kru88}) as well as criteria of metric subregularity based on \emph{outer coderivatives} (see
\cite{IofOut08,Kru15, ZheNg10.1,ZheNg10.2,ZheNg12}).
On the other hand, related studies in \cite[Theorem~7]{Kru09}, \cite[Theorem 5.1(ii)]{KruTha15}
have shown that regularity criteria developed for collections of
sets can be used when studying the corresponding properties of \SVM s.
The coderivatives (Fr\'echet, limiting or other) of the mappings \eqref{FAB} and \eqref{F3} employed in
Theorem~\ref{t:svm} admit simple representations in terms of the corresponding
normal cones to the sets involved in their definitions; see \cite[the proof of Theorem~8]{Kru09},
\cite[Lemma~5.1]{LewMal08},  \cite[P.~491]{LewLukMal09}, and \cite[Theorem~7.12 and Corollary~7.13]{Iof15}.
As a consequence, the coderivative criteria of regularity of \SVM s easily translate into the dual
characterizations of the corresponding regularity properties of collections of sets.
Not surprisingly, this way one rediscovers (some of) the known characterizations collected in
Theorem~\ref{t:svm};
see \cite[Theorem~8]{Kru09}, \cite[Theorem~5.1]{LewMal08},
and \cite[Theorems~7.12 and 7.15]{Iof15}.
\xqed
\end{remark}


The characterization of subtransversality given in Theorem \ref{t:MAINi}\eqref{t:MAINib}
and the definition of transversality shows that
transversality implies subtransversality (see Theorem \ref{t:tsr suff} below).
Alternatively, the implication is also immediate from Theorem \ref{t:MAINi}\eqref{t:MAINic} and
Theorem \ref{t:MAINii}\eqref{t:MAINiib}.
There are a number of other useful sufficient conditions for subtransversality, detailed
in the next theorem.

\begin{theorem}[sufficient conditions for subtransversality]\label{t:tsr suff}
If one of the following hold, then $\{A,B\}$ is subtransversal at $\xbar $.
\begin{enumerate}[\rm (i)]
\item\label{t:tsr suff 0} The collection $\{A,B\}$ is transversal at $\xbar $. Moreover,
$\rg\le\sr$.
\item\label{t:tsr suff iv}
There exist numbers $\alpha>0$ and $\delta>0$ such that ${\|v_1+v_2\|>\alpha}$
for all $x\in\B_{\de}(\bar x)$, $a\in A$, $b\in B$ with $0<\norm{x-a}<\de$, $0<\norm{x-b}<\de$,
and all nonzero $v_1\in\fncone{A}(a)$, $v_2\in\fncone{B}(b)$
satisfying
\begin{align*}
\|v_1\|+\|v_2\|=1,\quad
\frac{\ang{v_1,x-a}}{\|v_1\|\;\|x-a\|}>1-\de,\quad
\frac{\ang{v_2,x-b}}{\|v_2\|\;\|x-b\|}>1-\de.
\end{align*}
Moreover, $\sr\ge\al$.
%
\item\label{t:collection relations iii}
The sets $A$ and $B$ are \emph{intrinsically transversal} at $\xbar$  -- Definition \ref{D4}\eqref{D4:DIL2}.
\item\label{t:separable and hoelder}
The set $B$ intersects $A$ separably at $\xbar$ and $B$ is $0$-H\"older regular relative to $A$ at $\xbar$
with an adequate compromise between the constants.
\end{enumerate}
\end{theorem}
\begin{proof}
{\em Condition \eqref{t:tsr suff 0}.} This follows immediately from
Theorem \ref{t:MAINi}\eqref{t:MAINic} and
Theorem \ref{t:MAINii}\eqref{t:MAINiib}.\\

{\em Condition \eqref{t:tsr suff iv}. }
This result is new.  It expands and improves (in the setting adopted in the current article)
\cite[Theorem~4.1]{KruTha15}.  We do not have the space to provide the proof here, but refer
readers to the preprint \cite{KruLukTha}.\\

{\em \eqref{t:collection relations iii}}.
This was shown in \cite[Theorem~6.2]{DruIofLew15}.\\

\eqref{t:separable and hoelder}.
Suppose that $B$ intersects $A$ separably at $\xbar$ with constant $\alpha>0$ together
with a neighborhood $\Ball_{\delta}(\xbar)$ of $\xbar$
(Definition \ref{D4}\eqref{D4:NolRon}) and that $B$ is $0$-H\"older regular relative to $A$
at $\xbar$ with constant $c>0$ together with the same neighborhood $\Ball_{\delta}(\bx)$ of $\xbar$
(Definition \ref{d:set regularities}\eqref{d:Hoelder}).
Suppose that $\alpha+2c<1$.
We show that $\{A,B\}$ is subtransversal at $\bx$.
To see this, choose a number $\gamma$ satisfying
$\max\left\{\alpha+2c,\tfrac{1}{1+c^2}\right\} <\gamma<1$.
Thanks to Theorem \ref{t:MAINi}\eqref{t:MAINid},
it suffices to check the existence of a $\deltabar>0$ such that
\begin{equation}\label{t:21}
\frac{1-\gamma}{1+\gamma} d(a,A\cap B)\le d(a,B)\;\;\mbox{for all}\;\; a\in A\cap\B_{\deltabar}(\bx ).
\end{equation}
Let us first check that for any $a\in A\cap\B_{\delta}(\bx )$, $b\in P_B(a) \cap \B_{\delta}(\bx )$, and
$a_+\in P_A(b)\cap \B_{\delta}(\bx )$, one has
\begin{equation}\label{2}
\|a_+-b\| \le \gamma\|b-a\|.
\end{equation}
If $\|a_+-b\| \le \tfrac{1}{1+c^2}\|b-a\|$, then \eqref{2} holds true since $\tfrac{1}{1+c^2}\le \gamma$.
Suppose that $\|a_+-b\| > \tfrac{1}{1+c^2}\|b-a\|$.
The two regularity assumptions then yield the following inequalities:
\begin{align*}
\ip{a_+-b}{a-b} &\le \alpha\|a_+-b\|\|a-b\|,\\
\ip{b-a_+}{a-a_+} &\le c\|b-a_+\|\|a-a_+\|.
\end{align*}
Adding these inequalities and using $\|a-a_+\|\le 2\|a-b\|$, we get
$$
\|a_+-b\|\le (\alpha+2c)\|b-a\|\le \gamma\|b-a\|.
$$
Hence, \eqref{2} is proved.
Employing the basic routine originated in the proof of \cite[Theorem 5.2]{LewLukMal09}
implies the existence of a number $\deltabar>0$ such that for any $a\in A\cap\B_{\deltabar}(\bx )$,
there exists an $\tilde{x}\in A\cap B$ such that $\|a-\tilde{x}\|\le \tfrac{1+\gamma}{1-\gamma}\|a-b\|$.
Since $d(a,A\cap B)\le \|a-\tilde{x}\|$ and $\|a-b\|=d(a,B)$, condition \eqref{t:21} is proved.
%
\qed
\end{proof}

\begin{remark}
In light of the framework presented above, the transversality has been shown to be equivalent
to properties \eqref{D4:lin-reg} of Definition \ref{D4} and (in the case of smooth manifolds) \eqref{Tran-1},
and strictly to imply the other properties in Definition \ref{D4}.
\xqed
\end{remark}

\begin{remark}[entaglement of elemental regularity and regularity of collections of sets]
Proposition \ref{t:collection relations a}\eqref{t:collection relations va} and
Theorem \ref{t:tsr suff}\eqref{t:separable and hoelder} demonstrate that regularity of individual sets has implcations
for the regularity of the collection of sets.
The converse entaglement has also been observed in \cite[Proposition 8]{NolRon15}: if $A$ and $B$ are intrinsically
transversal at $\xbar$ with constant $\alpha$, then $A$ is $\sigma$-H\"older regular at $\xbar$ relative to $B$ for
every $\sigma\in [0,1)$ with any constant
$c< \frac{\alpha^2}{1-\alpha^2}$.

As a consequence of Proposition \ref{t:set regularity equivalences}\eqref{t:Hoelder char}, if
$A$ and $B$ 
are
intrinsically transversal at $\xbar$ with constant $\alpha\in (0,1]$ and, in addition,
there is a neighborhood $W$ of
$\xbar$ and a positive constant $\varepsilon<\frac{\alpha}{\sqrt{1-\alpha^2}}$
such that for each $(a,v)\in V$
defined in \eqref{e:Hoelder normals},
condition \eqref{e:xbar neighborhood} holds true,
then $A$ is elementally subregular of any order $\sigma\in[0,1)$ relative to
$ A\cap P^{-1}_B\paren{a+v}$ at $\xbar$
for each $(a,v)\in V$ with constant $\varepsilon$ and the respective neighborhood
$U(a,v)$.
\xqed
\end{remark}

\subsection{Special cases: convex sets, cones and manifolds}
A number of simplifications are possible in the convex setting, for cones and for manifolds.

The next representations follow from the simplified representations
for $\rg$ that are possible for convex sets or cones (cf. \cite[Propositions~13 and 15]{Kru05}).
\begin{proposition}[collections of convex sets]\label{P2-2}
Suppose $A$ and $B$ are convex.
The collection $\{A,B\}$ is transversal at $\xbar $ if and only if
one of the next two equivalent conditions holds true:
\begin{enumerate}
\item
there exists a number $\alpha>0$ such that
\begin{equation}\label{P2-2-1}
(A-x_1)\cap(B-x_2)\cap\B_{\rho}(\xbar )\neq \emptyset
\end{equation}
for all $\rho>0$ and all $x_1,x_2\in X$ with $\max\{\|x_1\|,\|x_2\|\}<\al\rho$;
\item
there exists a number $\alpha>0$ such that condition \eqref{P2-2-1} is satisfied
for some $\rho>0$ and all $x_1,x_2\in X$ with $\max\{\|x_1\|,\|x_2\|\}<\al\rho$.
\end{enumerate}
Moreover, the exact upper bound of all numbers $\alpha$ in any of the above conditions equals $\rg$.
\end{proposition}

\begin{proposition}[cones]\label{P2-3}
Suppose $A$ and $B$ are cones.
The collection $\{A,B\}$ is transversal at $\xbar $ if and only if
there exists a number $\alpha>0$ such that
\begin{equation*}\label{P2-3-1}
(A-a-x_1)\cap (B-b-x_2)\cap\B\neq \emptyset
\end{equation*}
for all $a\in A,\;b\in B$ and all $x_1,x_2\in X$ with $\max\{\|x_1\|,\|x_2\|\}<\al$.
Moreover, the exact upper bound of all numbers $\alpha$ in any of the above conditions equals ${\rm r}[A,B](0)$.
\end{proposition}

In the case when $A$ and $B$ are smooth manifolds, the interesting understanding
established in \cite[Theorem 5.2]{LewMal08} is easily deduced.
\begin{proposition}[manifolds]\label{t:manifolds}
Let $A$ and $B$ be smooth manifolds around a point $\xbar \in A\cap B$.
Then
\begin{align}\label{Th_9Feb}
{\rm r}_a[A,B](\bx)=c(A,B,\bx).
\end{align}
\end{proposition}
\begin{proof}
The equalities in Theorem \ref{t:MAINii}\eqref{t:MAINiig} and \eqref{t:MAINiij}
imply the following relation between $\rga$ and $\rgdd$:
\begin{align}\label{re}
\rga+(\rgdd)^2=1.
\end{align}
In the case when $A$ and $B$ are smooth manifolds, comparing the definition of $\rgdd$ with condition
\eqref{Fri2} (with $V_1=N_A(\bx)$ and $V_2=N_B(\bx)$) and taking into account the above equality
and definition \eqref{e:angle}, yields the equality \eqref{Th_9Feb}.
\qed
\end{proof}

\begin{remark}
 Theorem 5.2 of \cite{LewMal08} reduces to Proposition~\ref{t:manifolds}
 with the regularity estimate
being a direct consequence of the equality \eqref{re} in view of Proposition~\ref{t:manifolds}.
\xqed
\end{remark}

\begin{remark}
Some sufficient and also necessary characterizations of the subtransversality
property in terms of the Fr\'echet subdifferentials of the
function $x\mapsto d(x,A)+d(x,B)$ were formulated
\cite[Theorem~3.1]{NgaThe01}.
\xqed
\end{remark}

The next example illustrates the computation of the constants characterizing regularity.
\begin{example}\label{E5}
Let $\mathbb{E}=\R^2$, $A=\R\times\{0\}$, $B=\{(t,t)\mid t\in\R\}$, $\bx=(0,0)$.
$A$ and $B$ are linear subspaces.
We have $A\cap B=\{(0,0)\}$, $T_A(\bar{x})=A$, $T_B(\bar{x})=B$, $T_{A\cap B}(\bar{x}) =\{(0,0)\}$, $N_A(\bar{x})=A^{\perp}=\{0\}\times\R$, $N_B(\bar{x})=B^{\perp}=\{(t,-t)\mid t\in\R\}$.
Conditions~\eqref{Tran-1}, \eqref{BasQua-} and \eqref{BasQua} hold true.
The collection $\{A,B\}$ is transversal and, thanks to Proposition~\ref{t:manifolds},
transversal at $\bar{x}$.
By the representations in Theorem \ref{t:MAINii}\eqref{t:MAINiie-z}-\eqref{t:MAINiij},
after performing some simple computations, we obtain: \begin{align*}
\rg&=\frac{1}{2} \norm{\left(\frac{1}{\sqrt{2}}-1,\frac{1}{\sqrt{2}}\right)}= t_2,
\\
\rgd&=\frac{1}{2} \norm{\left(\frac{1}{\sqrt{2}}+1,\frac{1}{\sqrt{2}}\right)}= t_1,
\\
\rgdd&= \sqrt{d^2\left(\left(t_1,t_2\right),A\right)+ d^2\left(\left(t_1,t_2\right),B\right)}
\\
&= \sqrt{\norm{\left(t_1,t_2\right)- \left(t_1,0\right)}^2+ \norm{\left(t_1,t_2\right)- \left(\frac{t_1+t_2}{2},\frac{t_1+t_2}{2}\right)}^2}= t_2\sqrt{2},
\\
\rga&= \ang{\left(\frac{1}{\sqrt{2}},\frac{1}{\sqrt{2}}\right), (1,0)}=\frac{1}{\sqrt{2}},
\end{align*}
where $t_1:=\frac{\sqrt{2+\sqrt{2}}}{2}$ and $t_2:=\frac{\sqrt{2-\sqrt{2}}}{2}$.
It is easy to check that all the relations in Theorem \ref{t:MAINii}\eqref{t:MAINiie-z}-\eqref{t:MAINiij}
are satisfied.
\xqed
\end{example}

\section{Conclusion}
Our objective in this note has been to catalog the underlying theoretical tools behind the
recent flurry of activity on local convergence results for projection-type algorithms for
feasibility.  Of course the theory compiled here is not limited to feasibility problems, but
since most of the notions are geometrically motivated, feasibility offers the most vivid motivation
for the definitions.  Strong metric subregularity for set-valued mappings provides, in our opinion, the most powerful avenue
toward generalizing these ideas to constrained optimization and beyond.

We conclude this note with the underlying challenge that has inspired the development
of this theory:  what are {\em necessary conditions} for local linear
convergence of first order methods?

\end{document}